\documentclass{article}

\usepackage{amsmath, amssymb}
\usepackage{graphicx, wrapfig} 
\usepackage{amsthm}
\usepackage{url}
\usepackage{color}

\theoremstyle{plain}
\newtheorem{theorem}{Theorem}

\theoremstyle{definition}
\newtheorem{definition}{Definition}
\newtheorem{example}{Example}

\newtheorem{proposition}{Proposition}
\newtheorem{lemma}{Lemma}
\newtheorem{corollary}{Corollary}

\newtheorem{remark}{Remark}
\newtheorem{conjecture}{Conjecture}

\title{The CN matrix of a pure braid projection}
\author{Yuko Ozawa\thanks{Graduate School of Advanced Mathematical Sciences, Meiji University, Nakano, Tokyo, 164-8525, Japan. Email: yuko\_ozawa@meiji.ac.jp}, 
Ayaka Shimizu\thanks{Institute for Global Leadership, Ochanomizu University, Ohtsuka, Tokyo, 112-8610, Japan. Email: shimizu.ayaka@ocha.ac.jp} and 
Yoshiro Yaguchi\thanks{Maebashi Institute of Technology, Maebashi, Gunma, 371-0816, Japan. Email: y.yaguchi@maebashi-it.ac.jp}}
\date{\today}

\begin{document}

\maketitle

\begin{abstract}
The CN matrix of an $n$-braid projection $B$ is an $n \times n$ matrix such that each $(i,j)$ entry indicates the number of crossings between $i^{th}$ and $j^{th}$ strands of $B$. 
In this paper, several patterns of an $n \times n$ matrix to be a CN matrix are discussed, and the CN matrix of a pure 6-braid projection is characterized. 
As an application, the OU matrix of a pure 6-braid diagram and the crossing matrix of a positive pure 6-braid are also characterized. 
\end{abstract}

\section{Introduction}
\label{section-intro}

An {\it $n$-braid} ($n\in \mathbb{N}$) is a disjoint union of $n$ strands with two horizontal bars in $\mathbb{R}^3$ such that each strand has the initial point on the upper bar and the endpoint on the lower bar, and runs from the upper bar to the lower bar without returning. 
Any braid is represented by a {\it braid diagram} on $\mathbb{R}^2$ as shown in Figure \ref{fig-m-ex}. 
\begin{figure}[ht]
\centering
\includegraphics[width=8cm]{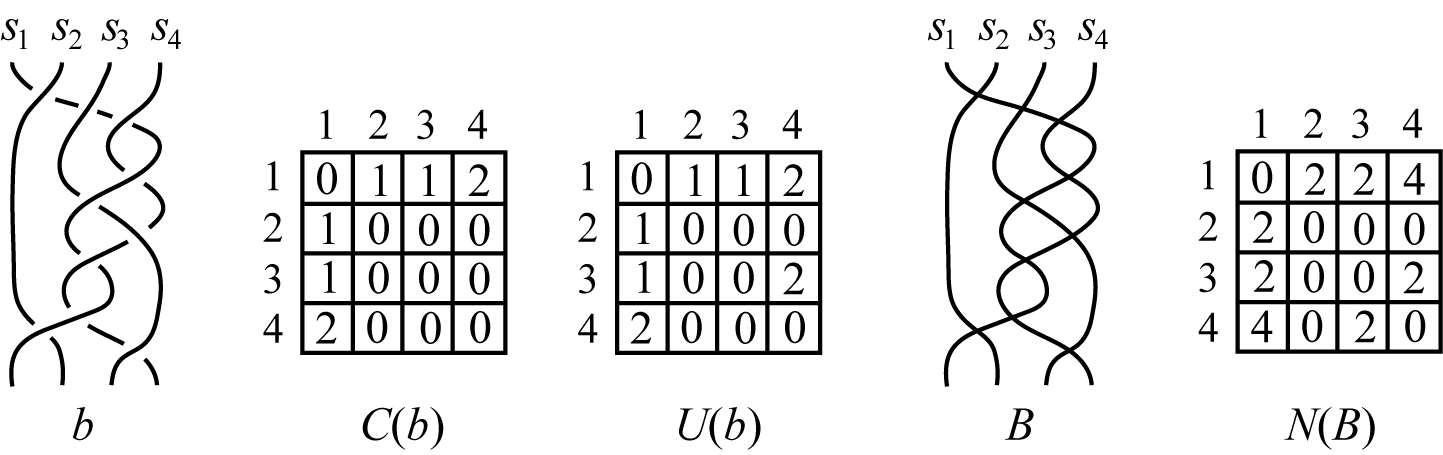}
\caption{The crossing matrix $C(b)$ and the OU matrix $U(b)$ of a 4-braid diagram $b$ and the CN matrix $N(B)$ of a 4-braid projection $B$. }
\label{fig-m-ex}
\end{figure}
Each crossing of a braid diagram $b$ has the sign as indicated in Figure \ref{fig-sign}\footnote{Our definition of the sign of a crossing of a braid diagram is opposite to that in \cite{Bu} and same to \cite{Gu, AY-5}. }. 
A braid diagram $b$ is said to be {\it positive} if all the crossings of $b$ are positive. 
A {\it positive braid} is a braid that has a positive diagram. 
Let $b$ be an $n$-braid diagram with strands $s_1, s_2, \dots , s_n$, where the indices of strands are given from left to right at the top of the braid diagram as shown in Figure \ref{fig-m-ex}. 
\begin{figure}[ht]
\centering
\includegraphics[width=2cm]{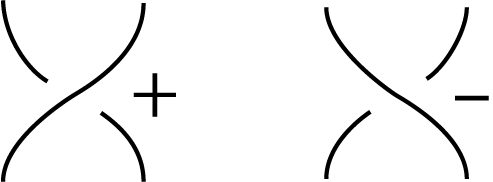}
\caption{A positive crossing on the left and a negative crossing on the right. }
\label{fig-sign}
\end{figure}
We say that a matrix $M$ is a {\it non-negative integer (resp. even) matrix} if every entry of $M$ is a non-negative integer (resp. even integer). 
The {\it crossing matrix} of a braid diagram was defined by Burillo, Gutierrez, Krsti\'{c} and Nitecki in \cite{Bu} as follows. 

\begin{definition}[\cite{Bu}]
The {\it crossing matrix} of a braid diagram $b$, denoted by $C(b)$, is an $n \times n$ zero-diagonal non-negative integer matrix such that the $(i,j)$ entry is the number of positive crossings minus the number of negative crossings between the strands $s_i$ and $s_j$ where $s_i$ is over $s_j$. (See Figure \ref{fig-m-ex}.) 
\end{definition}

\noindent (See also \cite{Gu}.) 
It was shown in \cite{Bu} that the crossing matrix does not depend on the choice of the diagram of a braid, namely, the crossing matrix is an invariant of the topological type of braids. 
We say that a braid or a braid diagram is {\it pure} if the endpoint of each strand $s_i$ is also at the $i^{th}$ position at the bottom. 
It is natural to study pure braids since the permutation of a pure braid is the identity, and they play an essential role in the study of the normal forms of braids (see, for example, \cite{Artin, El, th}). 
In \cite{Bu}, the following proposition was shown. 

\begin{proposition}[\cite{Bu}]
An $n \times n$ matrix $M$ is the crossing matrix of some pure $n$-braid diagram if and only if $M$ is a zero-diagonal symmetric integer matrix. 
\label{prop-T0-set}
\end{proposition}

\noindent We denote the $(i,j)$ entry of a matrix $M$ by $M(i,j)$. 

\begin{definition}[\cite{Bu}]
A zero-diagonal matrix $M$ is said to be T0 if whenever $1 \leq i<j<k \leq n$, then $M(i,j)=M(j,k)=0$ implies $M(i,k)=0$. 
\end{definition}

\noindent In this paper, a ``T0 matrix'' implies a ``T0 zero-diagonal matrix''. 
In \cite{Bu}, the following conjecture was proposed for positive pure braids. 

\begin{conjecture}[\cite{Bu}]
An $n \times n$ integer matrix $M$ is the crossing matrix of some positive pure $n$-braid diagram if and only if $M$ is a non-negative integer T0 symmetric matrix. 
\label{conj-C}
\end{conjecture}

\noindent Conjecture \ref{conj-C} was shown to be true for $n \leq 3$ by Burillo, Gutierrez, Krsti\'{c} and Nitecki in \cite{Bu} and for $n \leq 5$ by the second and third authors in \cite{AY-5}. 
The {\it OU matrix} was introduced by the second and third authors in \cite{AY-d} to describe the ``warping degree'' of a braid diagram, which estimates some knot invariants, including the unknotting number, of the closure, which is a knot or link diagram (\cite{ASA}). 

\begin{definition}[\cite{AY-d}]
The {\it OU matrix} of an $n$-braid diagram $b$, denoted by $U(b)$, is an $n \times n$ non-negative integer zero-diagonal matrix such that the $(i,j)$ entry is the number of crossings between $s_i$ and $s_j$ where $s_i$ is over $s_j$. (See Figure \ref{fig-m-ex}.) 
\end{definition}

\begin{proposition}[\cite{AY-5}]
If $b$ is a positive diagram, then $U(b)=C(b)$. 
\label{prop-C-OU}
\end{proposition}

\noindent Let $M^T$ denote the transpose of a matrix $M$. 
The OU matrix of a pure $n$-braid diagram was characterized for $n \leq 5$ in \cite{AY-5}, and based on the result, we propose the following conjecture. 

\begin{conjecture}
An $n \times n$ non-negative integer matrix $M$ is the OU matrix of some pure $n$-braid diagram if and only if $M+M^T$ is an even T0 matrix. 
\label{conj-OU}
\end{conjecture}

\noindent A {\it braid projection} is a braid diagram without over/under crossing information as illustrated in Figure \ref{fig-m-ex}. 
To address Conjectures \ref{conj-C} and \ref{conj-OU}, we explore the {\it CN matrix} which was introduced in \cite{AY-5}. 

\begin{definition}[\cite{AY-5}]
The {\it CN matrix} of an $n$-braid projection $B$, denoted by $N(B)$, is an $n \times n$ zero-diagonal matrix such that each $(i,j)$ entry is the number of crossings between $s_i$ and $s_j$. (See Figure \ref{fig-m-ex}.) 
\end{definition}

\noindent By definition, $N(B)$ is symmetric and $N(B)=U(b)+(U(b))^T$ holds for any braid diagram $b$ that has the projection $B$ (see Figure \ref{fig-m-ex}). 
We propose the following conjecture. 

\begin{conjecture}
An $n \times n$ non-negative integer matrix $M$ is the CN matrix of some pure $n$-braid projection if and only if $M$ is an even T0 symmetric matrix. 
\label{conj-CN}
\end{conjecture}

\noindent We show several patterns of an $n\times n$ matrix to be a CN matrix in Section \ref{section-formation}, and prove the following theorem in Section \ref{section-66}. 

\begin{theorem}
Conjectures \ref{conj-C}, \ref{conj-OU}, \ref{conj-CN} are true when $n \leq 6$. 
\label{thm-true6}
\end{theorem}

\noindent The structure of the paper is as follows. 
In Section \ref{section-prop-CN}, we review properties of the CN matrix. 
In Section \ref{section-BW}, we review the ``BW-ladder diagram'', a tool to test the possibility for a matrix to be a CN matrix. 
In Section \ref{section-formation}, we discuss several patterns of an $n \times n$ matrix to be a CN matrix. 
In Section \ref{section-conf}, we discuss the ``CN-realizable configurations'', a configuration of a matrix to realize a CN matrix. 
In Section \ref{section-66}, we discuss the relation of the crossing matrix, OU matrix and CN matrix, and prove Theorem \ref{thm-true6}.

\section{Properties of CN matrix}
\label{section-prop-CN}

In this section, we review the properties of the CN matrix. 
Let $B, C$ be $n$-braid projections. 
The {\it product of $B$ and $C$}, denoted by $BC$, is the $n$-braid projection that is obtained by connecting the $i^{th}$ endpoint of strand of $B$ and the $i^{th}$ initial point of $C$ for each $i=1, 2, \dots , n$. 

\begin{proposition}[\cite{AY-5}]
The addition $N(BC)=N(B)+N(C)$ holds when $B$ is a pure braid projection. 
\label{prop-pure-sum}
\end{proposition}

\begin{proposition}[\cite{AY-5}]
All the entries of the CN matrix $N(B)$ of a braid projection $B$ are even numbers if and only if $B$ is a pure braid projection. 
\label{prop-pure-even}
\end{proposition}

\begin{definition}
A matrix $M$ is said to be {\it CN-realizable} when $M$ is the CN matrix of some braid projection. 
\end{definition}

\begin{proposition}[\cite{AY-5}]
If a square matrix $M$ is CN-realizable, then $M$ is T0. 
\label{prop-T0}
\end{proposition}

\noindent For $n \leq 5$, Conjecture \ref{conj-CN} was shown to be true by the second and third authors. 

\begin{theorem}[\cite{AY-5}]
When $n \leq 5$, an $n \times n$ non-negative integer matrix $M$ is the CN matrix of some pure $n$-braid projection if and only if $M$ is an even T0 symmetric matrix. 
\label{thm-55}
\end{theorem}

\begin{definition}
For a square matrix $M$, the {\it reverse of $M$}, denoted by $M'$, is the matrix that is obtained from $M$ by reversing the order of the row and column. 
\end{definition}

\begin{proposition}[\cite{AY-5}]
If a matrix $M$ is CN-realizable, then $M'$ is also CN-realizable. 
\label{prop-reverse}
\end{proposition}

Since CN matrices are zero-diagonal symmetric matrices, we represent them by strictly upper triangular matrices for simplicity by replacing all the entries below the main diagonal with 0, until Section \ref{sub-sec-66}. 
A {\it $(0,2)$-matrix} is a matrix whose entries are 0 or 2. 
The following proposition was shown in \cite{AY-5}. 

\begin{proposition}[\cite{AY-5}]
Let $M$ be an $n \times n$ strictly upper triangular matrix whose entries are non-negative even numbers. 
Let $M^{02}$ be the $(0,2)$-matrix that is obtained from $M$ by replacing all the non-zero entries with 2. 
If $M^{02}$ is CN-realizable, then $M$ is also CN-realizable by a pure braid projection. 
\label{prop-M02}
\end{proposition}

\noindent By Propositions \ref{prop-pure-sum} and \ref{prop-pure-even}, we have the following proposition, which will be used in Section \ref{section-conf}. 

\begin{proposition}[\cite{AY-5}]
If $n \times n$ (0,2)-matrices $M_1$, $M_2$ are CN-realizable, then the sum $M_1 +M_2$ is also CN-realizable by a pure braid projection. 
\label{prop-02-sum}
\end{proposition}

\noindent We have the following proposition.

\begin{proposition}
Let $M_1$ be an $m \times m$ CN-realizable matrix. 
Let $M_2$ be an $n \times n$ matrix ($n \geq m$) such that 
\begin{align*}
\begin{cases}
M_2(i,j)=0 \text{ when } i<I \text{ or } j> I+m-1 \\
M_2(i+I-1, j+I-1) = M_1(i,j)
\end{cases}
\end{align*}
for some $1 \leq I \leq n-m+1$. 
Then $M_2$ is CN-realizable. 
\label{prop-small}
\end{proposition}

\begin{proof}
Let $B_1$ be an $m$-braid projection such that $N(B_1)=M_1$. 
Place $I-1$ strands on the left-hand side and $n-m-I+1$ strands on the right-hand side to obtain an $n$-braid projection $B_2$. 
Then $N(B_2)=M_2$. 
(See Figure \ref{fig-lem}.)
\end{proof}

\begin{figure}[ht]
\centering
\includegraphics[width=7cm]{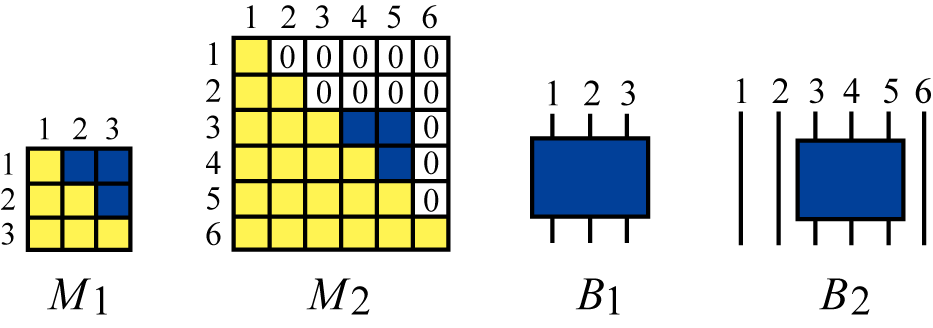}
\caption{An example of Proposition \ref{prop-small} with $m=3, n=6, I=3$. }
\label{fig-lem}
\end{figure}

\noindent For $6 \times 6$ matrices, we have the following lemma by Proposition \ref{prop-small} and Theorem \ref{thm-55}. 

\begin{lemma}
Let $M$ be a $6 \times 6$ T0 strictly upper triangular $(0,2)$-matrix such that $M(i,j)=0$ when $i<I$ or $j>I-m-1$ for some integers $1 \leq m \leq 5$ and $1 \leq I \leq 7-m$. 
Then $M$ is CN-realizable. 
\label{lem-15}
\end{lemma}

\section{BW-ladder diagram}
\label{section-BW}

In this section, we review the BW-ladder diagram which was introduced in \cite{AY-5} to test the CN-realizability of a matrix. 
A {\it BW-ladder diagram} is a ladder-fashioned diagram of $n$ vertical segments with two types of horizontal edges, a {\it black edge} (resp. {\it white edge}), $B^i_j$ (resp. $W^k_{k+1}$), with black vertices (resp. white vertices) on the endpoints on the $i^{th}$ and $j^{th}$ (resp. $k^{th}$ and $k+1^{th}$) segments. 
In particular, we call a BW-ladder diagram $D$ a {\it B-ladder diagram} (resp. {\it W-ladder diagram}) when $D$ has only black (resp. white) edges. 
Each BW-ladder diagram is represented by a sequence of $B^i_j$ and $W^k_{k+1}$ as well. 

\begin{definition}[\cite{AY-5}] 
Let $M$ be an $n \times n$ $(0,2)$-strictly upper triangular matrix. 
A {\it B-ladder diagram of $M$} is a B-ladder diagram with black edges $B^i_j$ for all the entries $M(i,j)=2$. 
\end{definition}

\noindent Here, as mentioned in \cite{AY-5}, a black edge $B$ implies a hook between the two strands where $B$ has the endpoints when we regard the $n$ segments as $n$ strands of a braid. 
Also, a white edge $W$ implies a crossing between the two strands where $W$ has the endpoints. 

\begin{definition}[\cite{AY-5}]
A {\it ladder move} is each of the local transformations on BW-ladder diagram depicted in Figure \ref{fig-l-move}, namely the following transformations.
\begin{itemize}
\item[(L1)] $B^i_{i+1} \leftrightarrow W^i_{i+1} W^i_{i+1}$  for any $i$. 
\item[(L2)] $B^i_j B^k_l \leftrightarrow B^k_l B^i_j$ for any $i<j, \ k<l$. 
\item[(L3)] $W^k_{k+1} W^l_{l+1} \leftrightarrow W^l_{l+1} W^k_{k+1}$ when $k+1<l$ or $l+1<k$. 
\item[(L4)] $W^i_{i+1} W^{i+1}_{i+2} W^i_{i+1} \leftrightarrow W^{i+1}_{i+2} W^i_{i+1} W^{i+1}_{i+2}$ for any $i$. 
\item[(L5)] $B^i_j W^k_{k+1} \leftrightarrow W^k_{k+1} B^i_j$ when $j<k$, $k+1<i$ or $i<k<k+1<j$.
\item[(L6)] $B^i_j W^i_{i+1} \leftrightarrow W^i_{i+1} B^{i+1}_j$ when $i+1 < j$.
\item[(L7)] $W^i_{i+1} B^i_j \leftrightarrow B^{i+1}_j W^i_{i+1}$ when $i+1 < j$.
\item[(L8)] $B^i_j W^{j-1}_j \leftrightarrow W^{j-1}_j B^i_{j-1}$ when $i < j-1$.
\item[(L9)] $W^{j-1}_j B^i_j \leftrightarrow B^i_{j-1} W^{j-1}_j$ when $i < j-1$.
\end{itemize}
\end{definition}
\begin{figure}[ht]
\centering
\includegraphics[width=12cm]{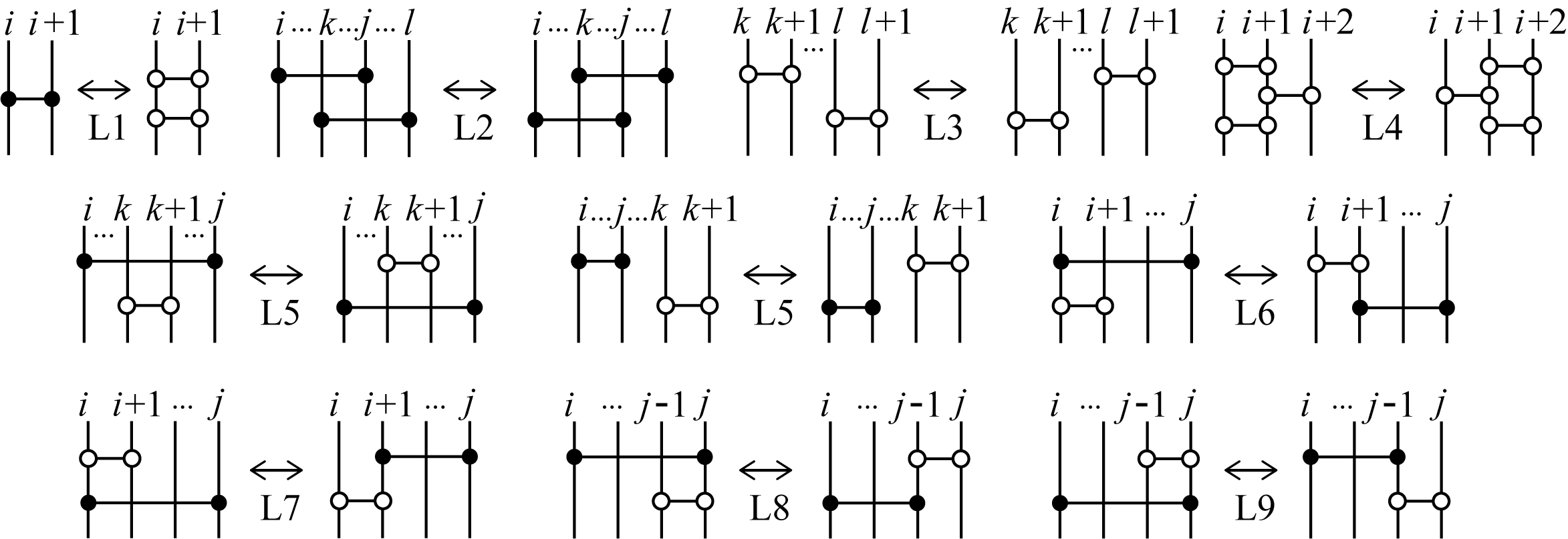}
\caption{Ladder moves. }
\label{fig-l-move}
\end{figure}

\begin{proposition}[\cite{AY-5}]
A strictly upper triangular $(0,2)$-matrix $M$ is CN-realizable if a B-ladder diagram of $M$ can be transformed into a W-ladder diagram by a finite sequence of ladder moves. 
\label{prop-W}
\end{proposition}

\begin{example}
The $(0,2)$-matrix $M$ in Figure \ref{fig-l-ex} is CN-realizable since the B-ladder diagram of $M$ is transformed into a W-ladder diagram by ladder moves. 
\begin{figure}[ht]
\centering
\includegraphics[width=10cm]{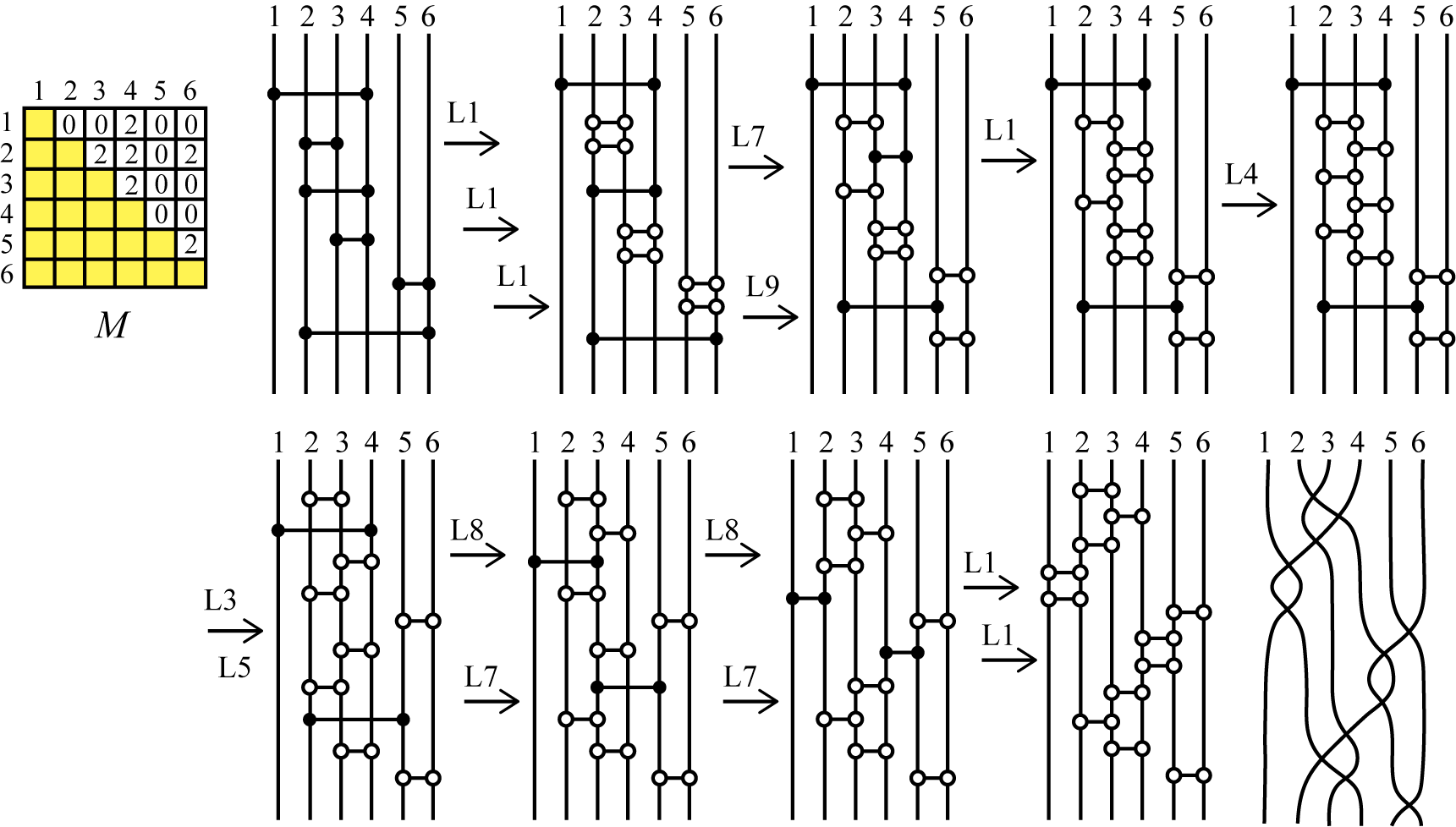}
\caption{The matrix $M$ is CN-realizable since the B-ladder diagram of $M$ is transformed into a W-ladder diagram. }
\label{fig-l-ex}
\end{figure}
\end{example}

\begin{proposition}[\cite{AY-5}]
The sequence $B^{k}_{k+1} B^{k}_{k+2} B^{k}_{k+3} \dots B^{k}_{l-1} B^{k}_{l}$ is equivalent to $W^{k}_{k+1} W^{k+1}_{k+2} W^{k+2}_{k+3} \dots W^{l-2}_{l-1}W^{l-1}_{l}W^{l-1}_{l}W^{l-2}_{l-1} \dots W^{k+2}_{k+3}W^{k+1}_{k+2}W^{k}_{k+1}$ up to ladder moves. 
The sequence $B^{l-1}_{l} B^{l-2}_{l} B^{l-3}_{l} \dots B^{k+1}_{l} B^{k}_{l}$ is equivalent to \\
$W^{l-1}_{l} W^{l-2}_{l-1} W^{l-3}_{l-2} \dots W^{k+1}_{k+2}W^{k}_{k+1}W^{k}_{k+1}W^{k+1}_{k+2} \dots W^{l-3}_{l-2}W^{l-2}_{l-1}W^{l-1}_{l}$ up to ladder moves. 
\label{prop-BtoW}
\end{proposition}

\noindent The following proposition was shown in \cite{AY-5} using the BW-ladder diagram. 

\begin{proposition}[\cite{AY-5}]
Any $n \times n$ T0 upper triangular $(0,2)$-matrix such that $M(i,j)=0$ for $j-i \geq 3$ is CN-realizable. 
\label{prop-3zero}
\end{proposition}

\section{CN-realizable formations}
\label{section-formation}

In this section, we show the following three patterns of an $n \times n$ CN-realizable matrices, the snake, hang-glider, and loupe formations in Sections \ref{subsec-s}, \ref{subsec-h}, and \ref{subsec-l}. 
We also discuss a general pattern, the T-structure, in Section \ref{subsec-T}.

\subsection{Snake formation}
\label{subsec-s}

\begin{definition}
We say that an $n \times n$ strictly upper triangular matrix $M$ has the {\it $r(k,J)$-formation} (or simply an {\it $r$-formation}) if $M$ satisfies the following conditions. 
\begin{align*}
\left\{
\begin{array}{ll}
M(k,j)=2 \text{ when } k+1 \leq j \leq J. \\
M(i,j)=0 \text{ otherwise. }
\end{array}
\right.
\end{align*}
We say that an $n \times n$ strictly upper triangular matrix $M$ has the {\it $c(I,l)$-formation} (or a {\it $c$-formation}) if $M$ satisfies the following conditions. 
\begin{align*}
\left\{
\begin{array}{ll}
M(i,l)=2 \text{ when } I \leq i \leq l-1. \\
M(i,j)=0 \text{ otherwise. } 
\end{array}
\right.
\end{align*}
\label{def-r}
\end{definition}

\noindent (See Figure \ref{fig-formations}.)
\begin{figure}[ht]
\centering
\includegraphics[width=7.5cm]{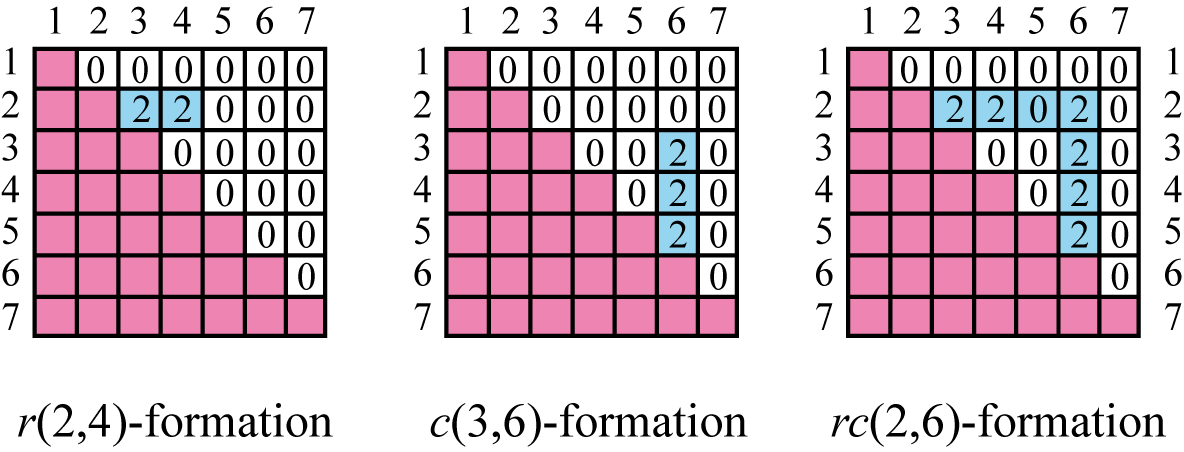}
\caption{The $r(2,4)$-, $c(3,6)$- and an $rc(2,6)$-formations. }
\label{fig-formations}
\end{figure}

\begin{proposition}
Any strictly upper triangular matrix of $r$- or $c$-formation is a CN-realizable matrix. 
\label{prop-r-form}
\end{proposition}

\begin{proof}
For the B-ladder diagram
$$D=B^k_{k+1} B^k_{k+2} \dots B^k_{J-1} B^k_{J}$$
of a matrix with an $r$-formation, we obtain the W-ladder diagram
\begin{align*}
W^{k}_{k+1} W^{k+1}_{k+2} \dots W^{J-2}_{J-1} W^{J-1}_{J} W^{J-1}_{J} W^{J-2}_{J-1} \dots W^{k+1}_{k+2} W^{k}_{k+1}
\end{align*}
by the ladder moves of Proposition \ref{prop-BtoW}. 
For the B-ladder diagram
$$B^{I}_l B^{I+1}_l \dots B^{l-2}_l B^{l-1}_l$$
of a matrix with a $c$-formation, we obtain the W-ladder diagram
\begin{align*}
W^{l-1}_{l} W^{l-2}_{l-1} \dots W^{I+1}_{I+2} W^{I}_{I+1} W^{I}_{I+1} W^{I+1}_{I+2} \dots W^{l-2}_{l-1} W^{l-1}_{l}
\end{align*}
by the ladder moves of Proposition \ref{prop-BtoW}. 
\end{proof}

\begin{definition}
Let $M$ be an $n \times n$ strictly upper triangular matrix which is the sum of $n \times n$ matrices $M_1 +M_2 + M_3$, where $M_1$ is the matrix with the $r(k,J)$-formation, $M_2$ is the matrix with the $c(I,l)$-formation, and $M_3$ is the matrix such that $M_3(k,l)=2$ and the other entries are zero. 
When $J<l$, $I>k$ and $I-J \leq 1$, we say that $M$ has an {\it $rc(k,l)$-formation} (or simply an {\it $rc$-formation}).
(See Figure \ref{fig-formations}.) 
\label{def-rc}
\end{definition}

\begin{proposition}
Any strictly upper triangular matrix $M$ of an $rc$-formation is a CN-realizable matrix. 
\label{prop-rc-form}
\end{proposition}

\begin{proof}
Let $M$ be an $n \times n$ matrix of an $rc(k,l)$-formation. 
Take the B-ladder diagram $D$ of $M$ so that $B^a_b$ is above $B^c_d$ if $a<c$ or $b<d$ for each pair of black edges $B^a_b$ and $B^c_d$ as shown in Figure \ref{fig-BW}, namely,  
$$D=B^k_{k+1} B^k_{k+2} \dots B^k_{J-1} B^k_{J} B^k_l B^{I}_l B^{I+1}_l \dots B^{l-2}_l B^{l-1}_l.$$
Apply ladder moves of Proposition \ref{prop-r-form} for the sequences of black edges before and after $B^k_l$ (the transformation $A$ in Figure \ref{fig-BW}). 
Then we obtain a BW-ladder diagram 
\begin{align*}
W^{k}_{k+1} W^{k+1}_{k+2} \dots W^{J-2}_{J-1} W^{J-1}_{J} W^{J-1}_{J} W^{J-2}_{J-1} \dots W^{k+1}_{k+2} W^{k}_{k+1} B^k_l & \\
W^{l-1}_{l} W^{l-2}_{l-1} \dots W^{I+1}_{I+2} W^{I}_{I+1} W^{I}_{I+1} W^{I+1}_{I+2} \dots W^{l-2}_{l-1} W^{l-1}_{l}. &
\end{align*}
We note that there are $J-k$ (resp. $l-I$) pairs of white edges before (resp. after) $B^k_l$ and $(J-k)+(l-I) \geq l-k-1$ because of the condition $I-J \leq 1$. 
Take $n, ~m \in \mathbb{N}$ so that $n+m=k-l-1$, $n \leq J-k$, and $m \leq l-I$. 
(In Figure \ref{fig-BW}, we take $n=m=2$.) 
Apply the ladder move L7 for the black edge $B^k_l$ with the $n$ white edges $W^k_{k+1}, W^{k+1}_{k+2}, \dots $ which are just above $B^k_l$. 
Then, apply L3, and apply L8 for the (shortened) black edge with the $m$ white edegs $W^{l-1}_l , W^{l-2}_{l-1}, \dots $ which are below the black edge. 
The black edge becomes length one, and by applying L1, we obtain a W-ladder diagram. 
\end{proof}

\begin{figure}[ht]
\centering
\includegraphics[width=12cm]{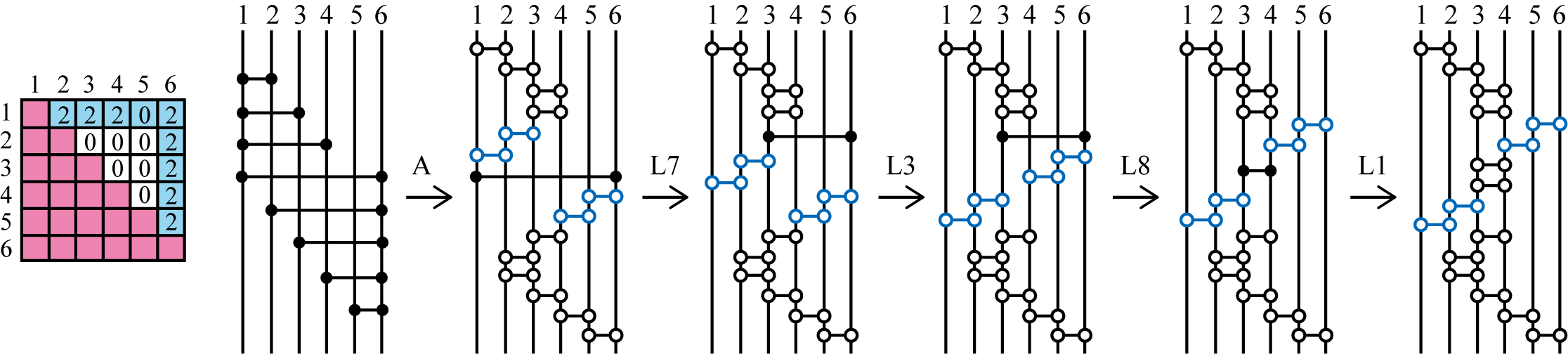}
\caption{A matrix of an $rc(1,6)$-formation is CN-realizable.}
\label{fig-BW}
\end{figure}

\begin{definition}
Let $( \alpha_1, \alpha_2 ) =(c,r),~(c, rc),~(rc, r)$ or $(rc, rc)$. 
An $n \times n$ strictly upper triangular matrix $M$ is said to have an $\alpha_1 (k,l)$-$\alpha_2 (l-1, m)$-formation (or simply $\alpha_1$-$\alpha_2$-formation) if $M=M_1+M_2-M_3$, where $M_1$ is a matrix of the $\alpha_1 (k,l)$-formation, $M_2$ is a matrix of the $\alpha_2(l-1, m)$-formation, and $M_3$ is the matrix such that $M_3(l-1, l)=2$ and the other entries are 0. 
(See Figure \ref{fig-cc-ex}.)
\label{def-c-r}
\end{definition}

\begin{figure}[ht]
\centering
\includegraphics[width=10cm]{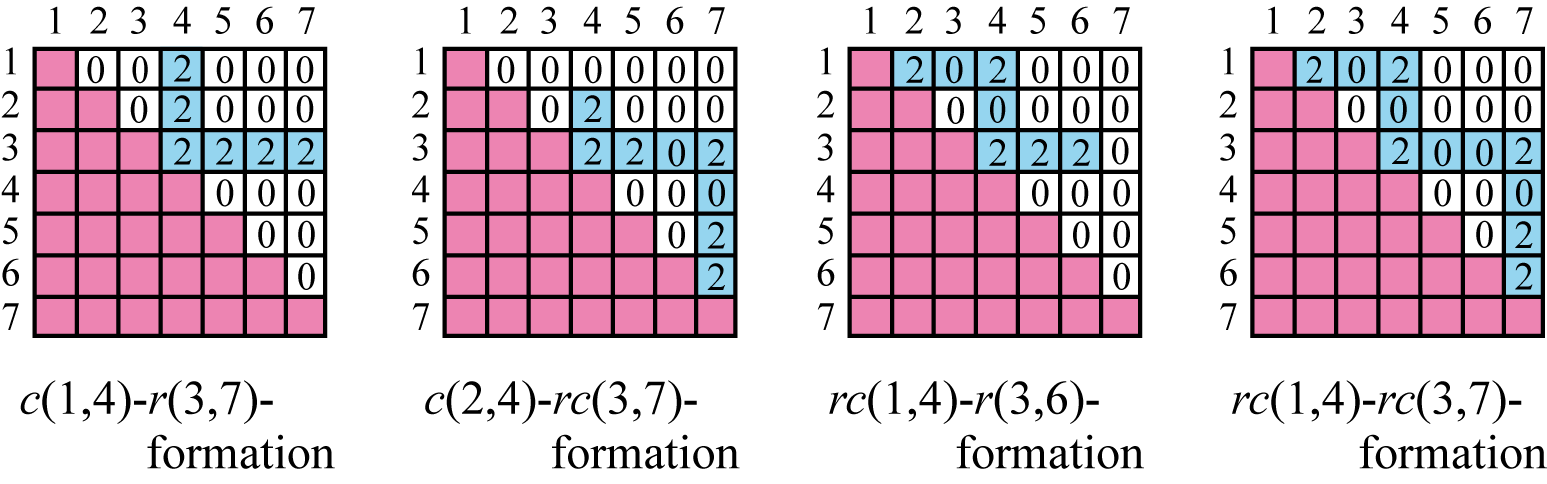}
\caption{$c$-$r$-, $c$-$rc$-, $rc$-$r$- and $rc$-$rc$-formations.}
\label{fig-cc-ex}
\end{figure}

\begin{proposition}
Any strictly upper triangular matrix of an $\alpha_1$-$\alpha_2$-formation is a CN-realizable matrix, where $( \alpha_1, \alpha_2 ) =(c,r),~(c, rc),~(rc, r)$ or $(rc, rc)$. 
\label{prop-alpha}
\end{proposition}

\begin{proof}
Let $M$ be a matrix of the $c(k,l)$-$rc(l-1,m)$-formation. 
Take a B-ladder diagram $D$ of $M$ as 
$$D=B^k_l B^{k+1}_l B^{k+2}_l \dots B^{l-1}_l B^{l-1}_{l+1} B^{l-1}_{l+2} \dots B^{l-1}_{J-1} B^{l-1}_J B^{l-1}_m B^{I}_m B^{I+1}_m \dots B^{m-2}_m B^{m-1}_m.$$
For the upper part from $B^k_l$ to $B^{l-1}_l$, apply the ladder moves of the proof of Proposition \ref{prop-r-form} (the transformation $A$ in Figure \ref{fig-cc-pf}) to obtain 
$$W^{l-1}_l W^{l-2}_{l-1} W^{l-3}_{l-2} \dots W^k_{k+1} W^k_{k+1} \dots W^{l-3}_{l-2} W^{l-2}_{l-1} W^{l-1}_l.$$
Then, for the lower part, including the lower white edge $W^{l-1}_{l}$, apply the ladder moves in the same manner to the proof of Proposition \ref{prop-rc-form} (the transformations $B$ in Figure \ref{fig-cc-pf}) to obtain white edges. 
Thus, we obtain a W-ladder diagram. 
For the other cases, $c$-$r$-, $rc$-$r$-, $rc$-$rc$-formations, W-ladder diagrams can be obtained in the same way. 
\end{proof}

\begin{figure}[ht]
\centering
\includegraphics[width=12cm]{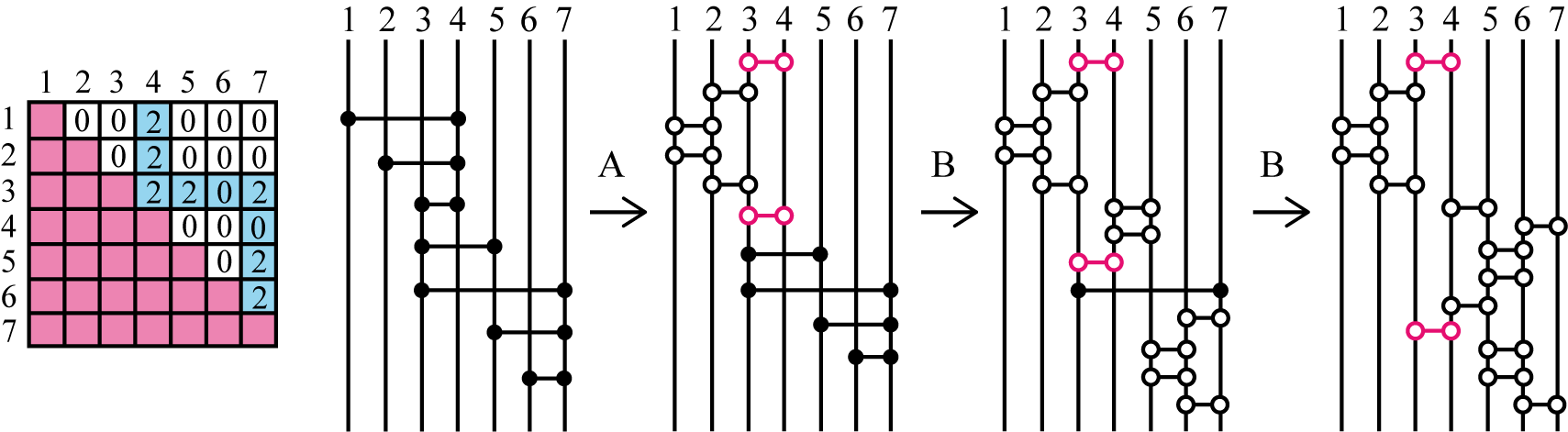}
\caption{A matrix of a $c(1,4)$-$rc(3,7)$-formation is a CN-realizable matrix. }
\label{fig-cc-pf}
\end{figure}

\begin{proposition}
Let $M_1$ (resp. $M_2$) be an $n \times n$ matrix such that $M_1(l-2, l+1)=2$ (resp. $M_2(l-2, l+2)=2$) and the other entries are zero.  
\begin{itemize}
\item[(1)] Let $M$ be an $n \times n$ matrix of a $c(k,l)$-$r(l-1,m)$-formation with $l-k \geq 2$, $m-l \geq 1$. The matrix $M+M_1$ is also a CN-realizable matrix. 
\item[(2)] Let $M$ be an $n \times n$ matrix of a $c(k,l)$-$r(l-1,m)$-formation with $l-k \geq 2$, $m-l \geq 2$. The matrix $M+M_1+M_2$ is also a CN-realizable matrix. 
\end{itemize}
\label{prop-sharp}
\end{proposition}

\noindent We call such a formation of the matrix $M+M_1$ or $M+M_1+M_2$ in Proposition \ref{prop-sharp} a {\it $c(k,l)\# r(l-1,m)$-formation} (or simply a {\it $c \# r$-formation}). \\

\noindent {\it Proof of Proposition \ref{prop-sharp}}. \ 
Let $D=B^{k}_{l} B^{k+1}_{l} \dots B^{l-1}_{l} B^{l-2}_{l+1} (B^{l-2}_{l+2}) B^{l-1}_{l+1} B^{l-1}_{l+2} \dots B^{l-1}_{m}$. 
For the upper part $B^{k}_{l} B^{k+1}_{l} \dots B^{l-1}_{l}$, apply the ladder moves of the proof of Proposition \ref{prop-r-form} to obtain $W^{l-1}_{l} W^{l-2}_{l-1} \dots W^{l-2}_{l-1} W^{l-1}_{l}$ (the transformation $A$ in Figure \ref{fig-sh-pf}). 
Use the ladder move L5 to move down the lower $W^{l-1}_l$ below $B^{l-2}_{l+1}$ (and $B^{l-2}_{l+2}$). 
Then transform the lower part of the BW-ladder diagram in the same way to the proof of Proposition \ref{prop-alpha} (the transformation $B$ in Figure \ref{fig-sh-pf}). 
Use the white edges $W^{l-2}_{l-1}$ and $W^l_{l+1}$ to shorten $B^{l-2}_{l+1}$, and use $W^{l-2}_{l-1}$, $W^{l-1}_l$ and $W^{l+1}_{l+2}$ to shorten $B^{l-2}_{l+2}$ by the ladder moves as shown in Figure \ref{fig-sh-pf}. 
\qed \\

\begin{figure}[ht]
\centering
\includegraphics[width=12cm]{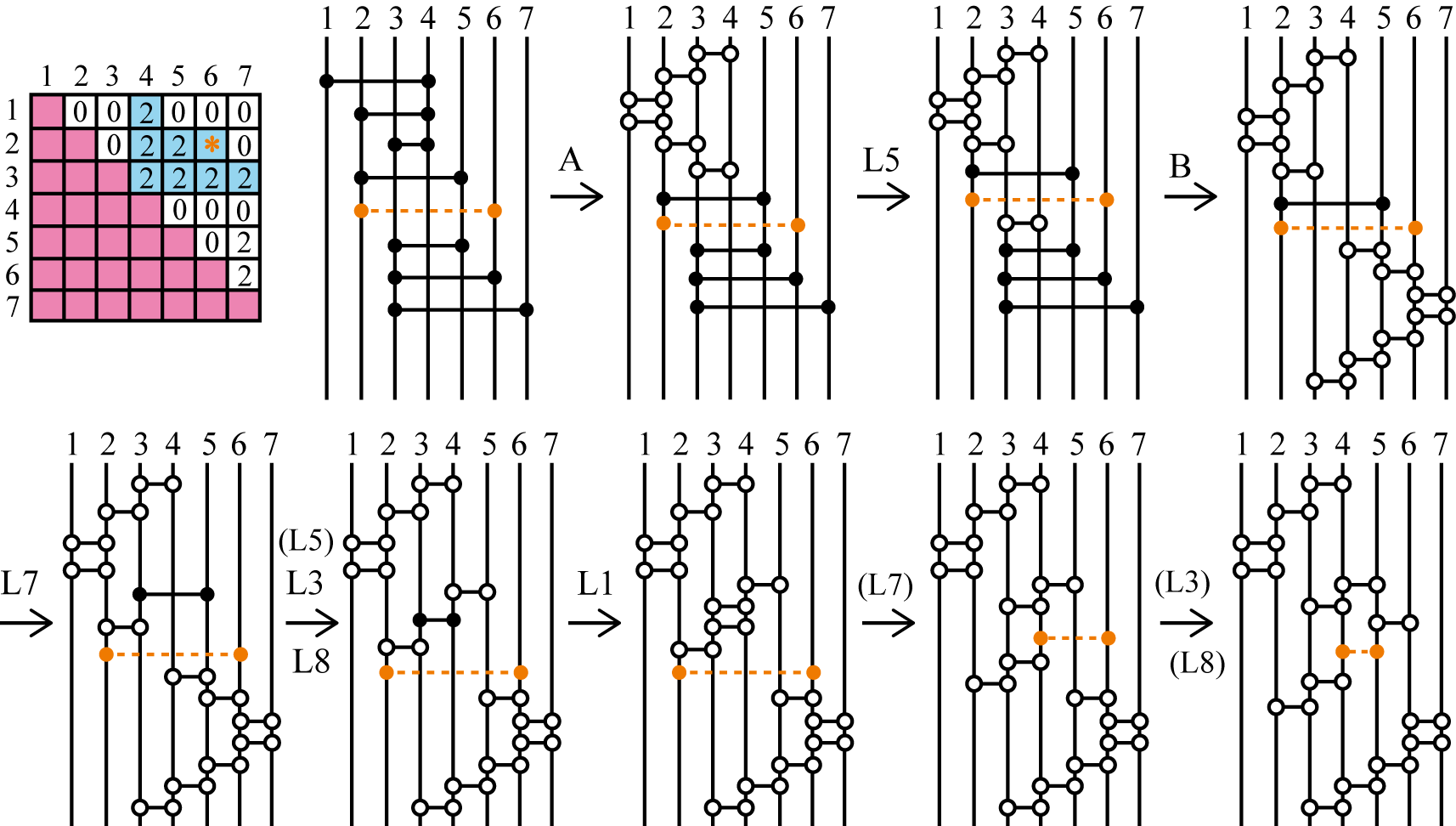}
\caption{A matrix of a $c(1,4) \# r(3,7)$-formation is a CN-realizable matrix. Ignore the orange-colored broken edges if $M(2,6)=0$.}
\label{fig-sh-pf}
\end{figure}

\noindent We call a formation of Definiton \ref{def-r}, \ref{def-rc}, or \ref{def-c-r} or Proposition \ref{prop-sharp} a {\it snake formation}.

\subsection{Hang-glider formation}
\label{subsec-h}

\begin{definition}
A strictly upper triangular matrix $M$ is said to have an {\it $H(k,l;m)$-formation} (or simply an {\it $H$-formation}) for $k<m<l$ if $M$ satisfies 
\begin{align*}
\left\{
\begin{array}{ll}
M(k, m+1)=M(k,l)=M(m,l)=M(m,m+1)=2. \\
M(k,j)=2 \text{ when } j \leq m-1. \\
M(i,l)=2 \text{ when } i \geq m+2. \\
M(k,m)=0 \text{ or } 2. \
M(m+1,l)=0 \text{ or } 2. \\
M(m-1, m+1)=0 \text{ or } 2. \
M(m, m+2)=0 \text{ or } 2. \\
M(i,j)=0 \text{ otherwise}. 
\end{array}
\right.
\end{align*}
\label{def-h}
\end{definition}

\begin{proposition}
Any strictly upper triangular matrix of an $H$-formation is a CN-realizable matrix. 
\label{prop-hang}
\end{proposition}

\begin{proof}
Let $M$ be an $n \times n$ matrix of an $H(k,l;m)$-formation. 
Take a B-ladder diagram $D$ of $M$ as 
$$B^k_l \ B^k_{k+1}B^k_{k+2} \dots B^k_{m-1} (B^k_{m}) B^k_{m+1}\ (B^{m-1}_{m+1}) B^m_{m+1} (B^m_{m+2}) \ B^{m}_l (B^{m+1}_l) B^{m+2}_l \dots B^{l-2}_l B^{l-1}_l.$$
Since $M$ has the $rc(l,m+1)$-$rc(m,l)$-formation if the $(k,l)$ entry is replaced with 0, keeping $B^k_l$ at the top, $D$ is transformed into 
\begin{align*}
& B^k_k \ W^k_{k+1} \dots (W^{m-1}_m W^{m-1}_m) W^{m}_{m+1} W^{m-1}_{m} W^{m-1}_{m} \dots W^{k}_{k+1} (W^{m-1}_m W^{m-1}_m )\\
& (W^{m+1}_{m+2} W^{m+1}_{m+2}) W^{l-1}_{l} \dots W^{m+1}_{m+2} W^{m+1}_{m+2} W^{m}_{m+1} (W^{m+1}_{m+2} W^{m+1}_{m+2}) \dots W^{l-1}_{l}
\end{align*}
by the ladder moves of the proof of Proposition \ref{prop-alpha} (the transformations $A$ in Figure \ref{fig-rc-ex}). 
Apply ladder moves L6, (L7 if $M(k,m)=2$) and L5 for the black edge $B^k_l$ with the white edges $W^k_{k+1} \dots (W^{m-1}_{m} W^{m-1}_{m} ) W^{m}_{m+1} W^{m-1}_{m}$. 
Then $B^k_l$ is shortened to $B^m_l$. 
After some L3 (and L5 if $M(m,m+2)=2)$, apply L8 for $B^m_l$ with the white edges $W^{l-1}_{l} \dots W^{m+1}_{m+2}$ to obtain $B^m_{m+1}$. 
By L1, then, we obtain a W-ladder diagram. 
\end{proof}

\begin{figure}[ht]
\centering
\includegraphics[width=13cm]{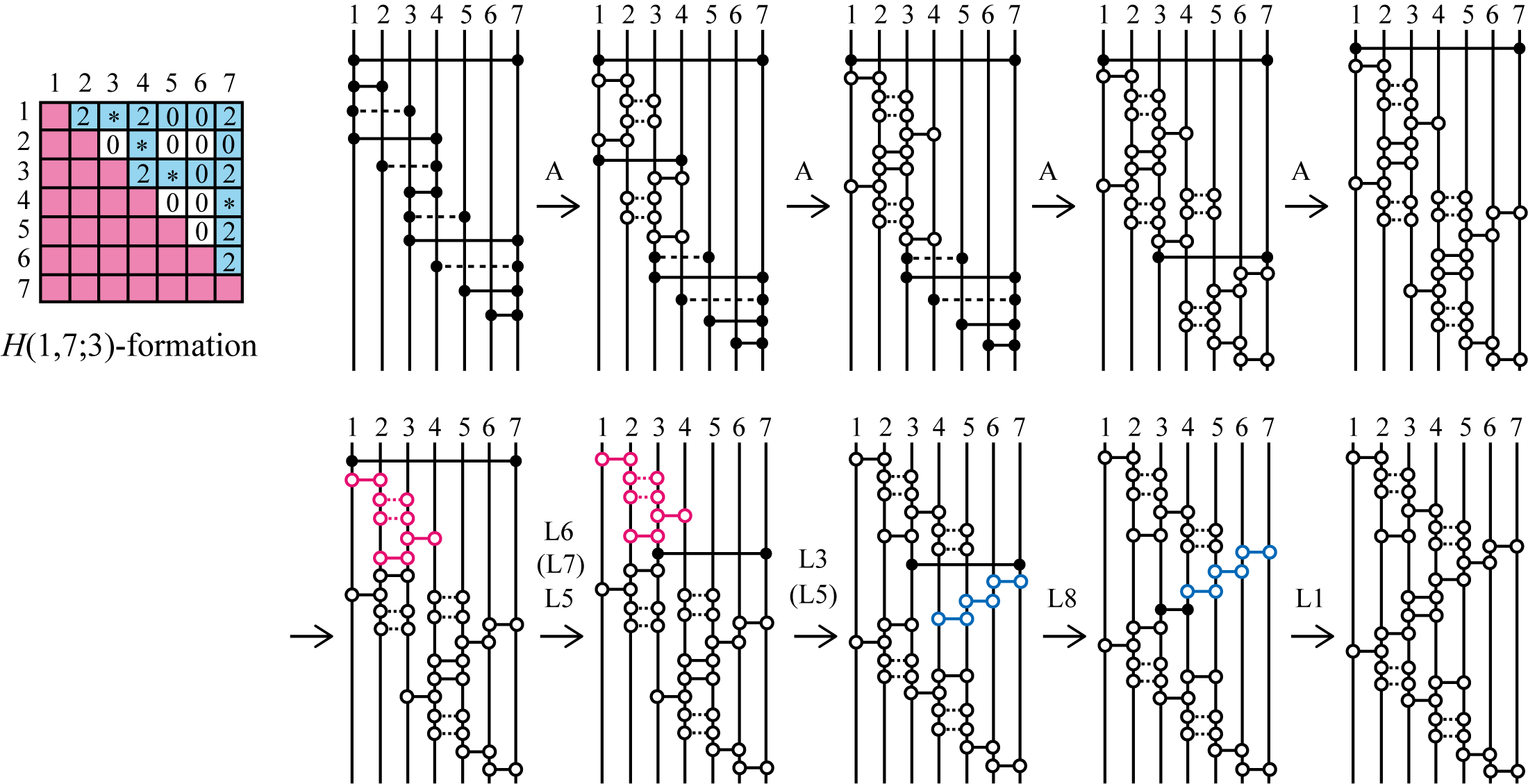}
\caption{A matrix of an $H(1,7;3)$-formation is a CN-realizable matrix. Ignore the broken edges if the corresponding entries are 0 in the matrix.}
\label{fig-rc-ex}
\end{figure}

\begin{corollary}
Let $M$ be an $n \times n$ strictly upper triangular matrix that has an $H(k,l;m)$-formation. 
Let $M_1$ (resp. $M_2$) be the $n \times n$ matrix such that $M_1(k-1, k+1)=2$ (resp. $M_2(l-1, l+1)=2$) and the other entries are 0. 
The matrices $M+M_1$, $M+M_2$, and $M+M_1+M_2$ are also CN-realizable. 
\label{cor-h}
\end{corollary}

\begin{proof}
For the B-ladder diagram in the proof of Proposition \ref{prop-hang}, add $B^{k-1}_{k+1}$ (resp. $B^{l-1}_{l+1}$) to the top (resp. to the bottom), and use the white edge $W{k}_{k+1}$ (resp. $W^{l-1}_l$) to shorten it. 
\end{proof}

\noindent We call a formation of Definition \ref{def-h} or Corollary \ref{cor-h} a {\it hang-glider formation.}

\subsection{Loupe formation}
\label{subsec-l}

\begin{definition}
A strictly upper triangular matrix $M$ is said to have an {\it $L_1 (k,l)$-formation} (or simply an {\it $L_1$-formation}) if $M$ satisfies 
\begin{align*}
\left\{
\begin{array}{ll}
M(k,j)=2 \text{ when } j=k+1 \text{ or } k+3 \leq j \leq l. \\
M(k,k+2)=0 \text{ or } 2. \\
M(k+1,k+3)=0 \text{ or } 2. \\
M(k+2,j)=2 \text{ when } k+3 \leq j \leq l. \\
M(i,j)=0 \text{ otherwise}. 
\end{array}
\right.
\end{align*}
\label{def-l1}
\end{definition}

\begin{proposition}
Any strictly upper triangular matrix of an $L_1$-formation is a CN-realizable matrix. 
\label{prop-loupe}
\end{proposition}

\begin{proof}
Let $M$ be an $n \times n$ matrix of an $L_1(k,l)$-formation. \\
Let $S^k_l= B^k_{k+1} B^k_{k+3} B^k_{k+4} \dots B^k_{l-1} B^k_{l} B^{k+2}_{l} B^{k+2}_{l-1} \dots B^{k+2}_{k+4} B^{k+2}_{k+3}$, and take a B-ladder diagram $D$ of $M$ as $D=(B^k_{k+2}) S^k_l (B^{k+1}_{k+3})$. 
Apply ladder moves L1 for the black edges $B^k_{k+1}$ and $B^{k+2}_{k+3}$ to obtain $W^k_{k+1}W^k_{k+1}$ and $W^{k+2}_{k+3}W^{k+2}_{k+3}$. 
Apply L7, L5, L3 to move down the lower white edge $W^k_{k+1}$, and apply L6, L5, L8 to move up the above white edge $W^{k+2}_{k+3}$. 
Then we obtain 
$$(B^k_{k+2}) W^k_{k+1} W^{k+1}_{k+2} S^{k+1}_l W^{k+2}_{k+3} W^k_{k+1} (B^{k+1}_{k+3}).$$ 
Repeat the procedure until we obtain $S^{l-3}_l$. 
Then, by the ladder moves L1, L7, L3, L8, and L1, we obtain 
$$(B^k_{k+2}) W^k_{k+1} W^{k+2}_{k+3}  W^{k+1}_{k+2}  W^{k+3}_{k+4} \dots  W^{l-1}_{l} W^{l-2}_{l-1} W^{l-2}_{l-1}  W^{l-1}_{l} \dots  W^{k+3}_{k+4}  W^{k+1}_{k+2} W^{k+2}_{k+3} W^k_{k+1} (B^{k+1}_{k+3})$$
as shown in Figure \ref{fig-lou}. 
Apply L3 for the pair of white edges at the bottom to obtain
$$(B^k_{k+2}) W^k_{k+1} W^{k+2}_{k+3}  W^{k+1}_{k+2}  W^{k+3}_{k+4} \dots  W^{l-1}_{l} W^{l-2}_{l-1} W^{l-2}_{l-1}  W^{l-1}_{l} \dots  W^{k+3}_{k+4}  W^{k+1}_{k+2} W^{k}_{k+1} W^{k+2}_{k+3} (B^{k+1}_{k+3}).$$
If $M(k, k+2)$ or $M(k+1, k+3)$ is 2, apply L6 or L9 as shown in Figure \ref{fig-lou} to obtain a W-ladder diagram. 
\end{proof}
\begin{figure}[ht]
\centering
\includegraphics[width=12cm]{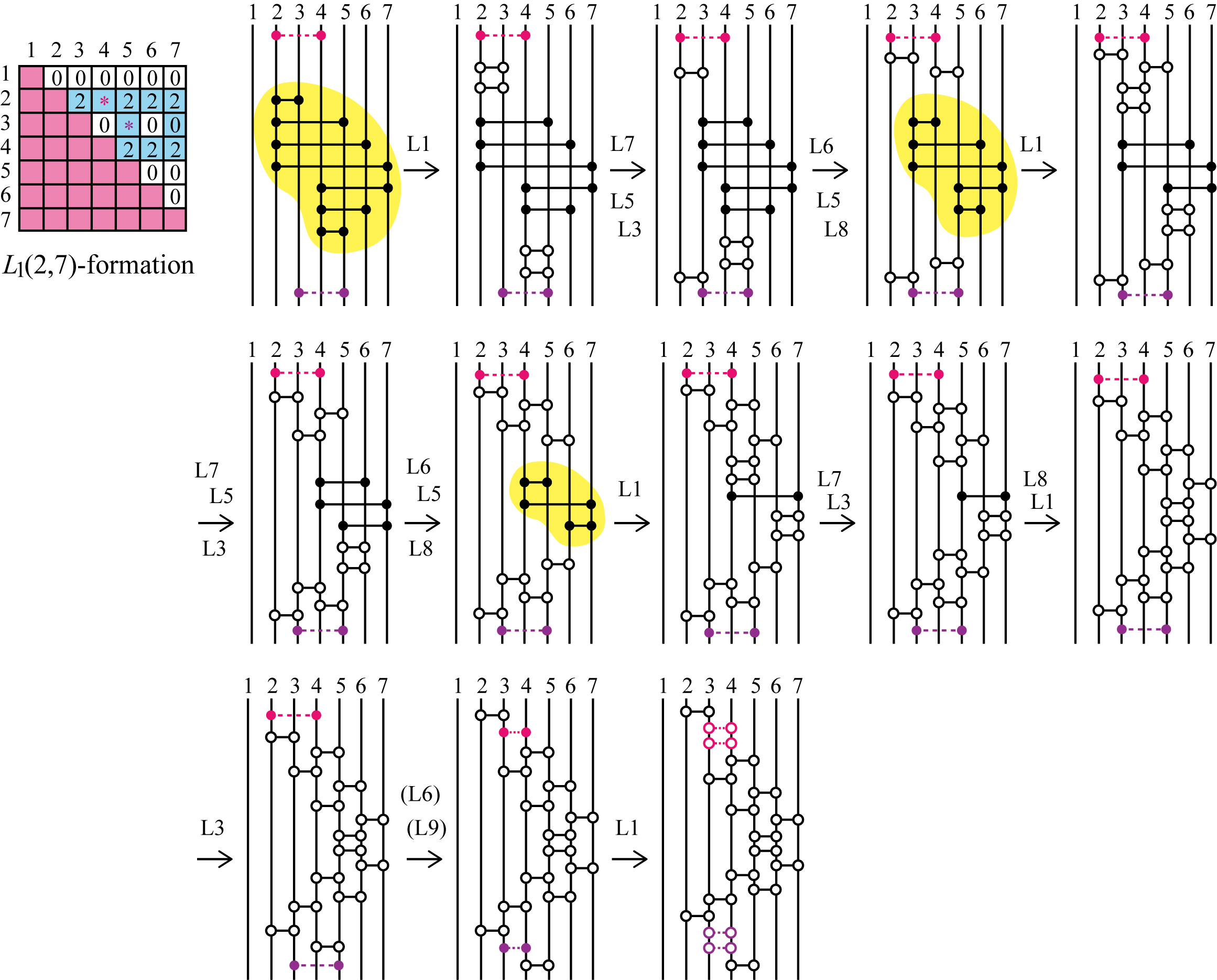}
\caption{A matrix of an $L_1(2,7)$-formation is a CN-realizable matrix. $S^2_7$, $S^3_7$, $S^4_7$ are highlighted in yellow. }
\label{fig-lou}
\end{figure}

\begin{corollary}
Let $M$ be an $n \times n$ strictly upper triangular matrix that has an $L_1(k,l)$-formation. 
Let $M_1$ (resp. $M_2$) be the $n \times n$ matrix such that $M_1(k-1, k+1)=2$ (resp. $M_2(k+2, l+1)=2$) and the other entries are 0. 
The matrices $M+M_1$, $M+M_2$, $M+M_1+M_2$ are also CN-realizable. 
\label{cor-l1}
\end{corollary}

\begin{proof}
For the B-ladder diagram in the proof of Proposition \ref{prop-loupe}, add $B^{k-1}_{k+1}$ to the top, and $B^{k+2}_{l+1}$ to the bottom. 
Use the white edge $W^k_{k+1}$ for $B^{k-1}_{k+1}$, the edges $W^{k+2}_{k+3},~W^{k+3}_{k+4}, \dots , W^{l-1}_l$ for $B^{k+2}_{l+1}$ to shorten them and apply L1 to transform them into white edges.  
\end{proof}

\begin{definition}
A strictly upper triangular matrix $M$ is said to have an {\it $L_2 (k,l)$-formation} (or simply an {\it $L_2$-formation}) if $M$ satisfies 
\begin{align*}
\left\{
\begin{array}{ll}
M(k,j)=2 \text{ when } j=k+1, k+2 \text{ or } k+4 \leq j \leq l. \\
M(k,k+3)=0 \text{ or } 2. \\
M(k+2,k+4)=0 \text{ or } 2. \\
M(k+3,j)=2 \text{ when } k+4 \leq j \leq l. \\
M(i,j)=0 \text{ otherwise}. 
\end{array}
\right.
\end{align*}
\label{def-l2}
\end{definition}

\begin{proposition}
Any strictly upper triangular matrix of an $L_2$-formation is a CN-realizable matrix. 
\label{prop-loupe2}
\end{proposition}

\begin{proof}
Let $S^k_l=B^k_{k+1} B^k_{k+2} B^k_{k+4} B^k_{k+5} \dots B^k_{l-1} B^k_l B^{k+3}_{l} B^{k+3}_{l-1} \dots B^{k+3}_{k+5} B^{k+3}_{k+4}$. 
Then $S^k_l$ is transformed into $W^{k}_{k+1} W^{k+3}_{k+4} S^{k+1}_l W^{k+3}_{k+4} W^{k}_{k+1}$ by the ladder moves as shown in Figure \ref{fig-lou2}. 
Let $M$ be an $n \times n$ matrix of an $L_2(k,l)$-formation. 
Take a B-ladder diagram $D$ of $M$ as $(B^k_{k+3}) S^k_l (B^{k+2}_{k+4})$. 
Apply the ladder moves until we obtain $S^{l-4}_l$, which is transformed into white edges by ladder moves as shown in Figure \ref{fig-lou2}. 
If $M(k, k+3)$ or $M(k+2, k+4)$ is 2, use the first half sequence of white edges or the white edge $W^{k+3}_{k+4}$ to shorten them. 
\end{proof}
\begin{figure}[ht]
\centering
\includegraphics[width=11cm]{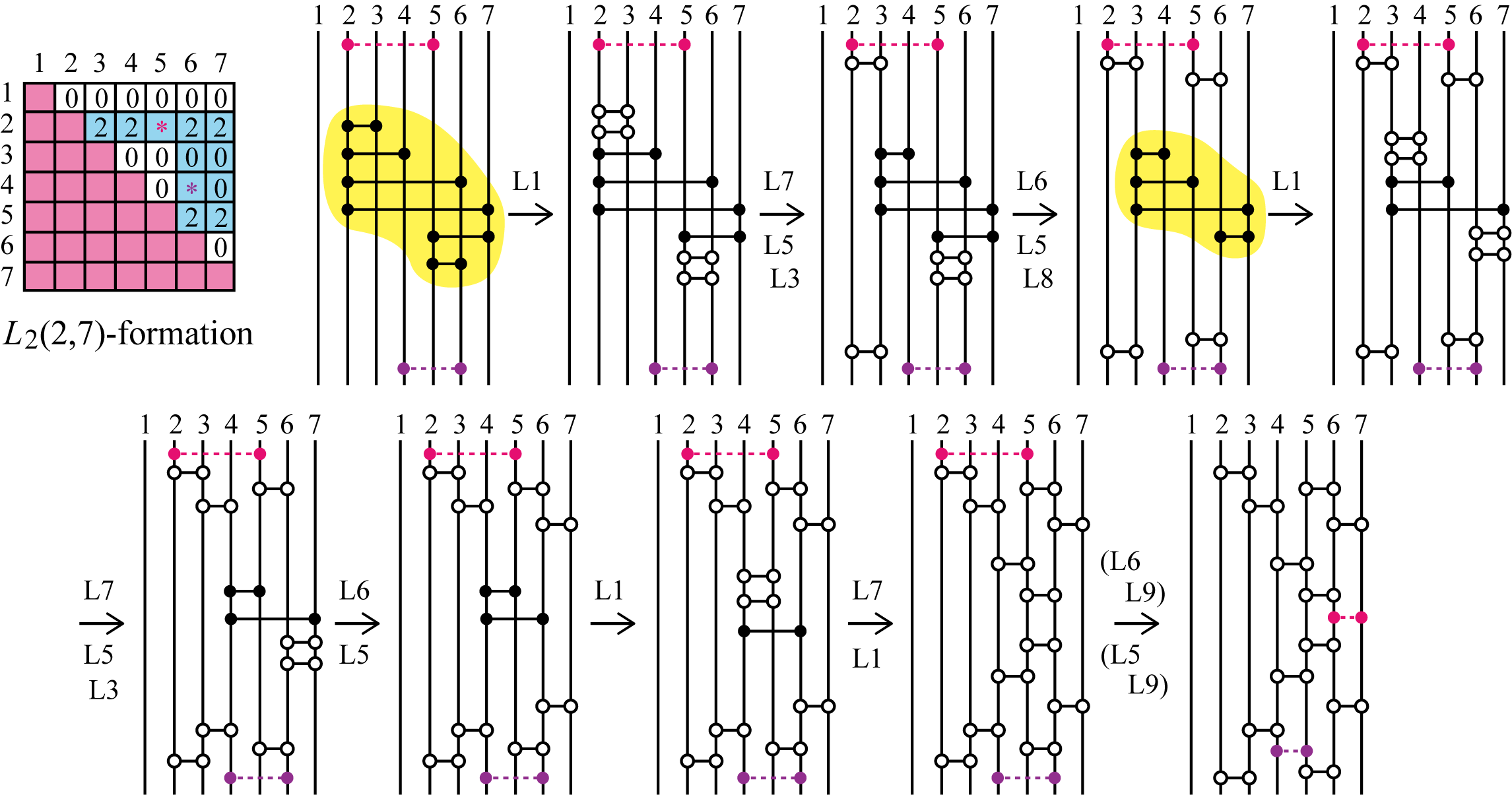}
\caption{A matrix of an $L_2(2,7)$-formation is a CN-realizable matrix. $S^2_7$ and $S^3_7$ are highlighted in yellow. }
\label{fig-lou2}
\end{figure}

\begin{definition}
A strictly upper triangular matrix $M$ is said to have an {\it $L_3 (k,l)$-formation} (or simply an {\it $L_3$-formation}) if $M$ satisfies 
\begin{align*}
\left\{
\begin{array}{ll}
M(k,j)=2 \text{ when } j=k+1 \text{ or } k+4 \leq j \leq l. \\
M(k,k+2)=0 \text{ or } 2. \\
M(k+1,k+4)=0 \text{ or } 2. \\
M(k+2,j)=2 \text{ when } k+4 \leq j \leq l. \\
M(k+3,j)=2 \text{ when } k+4 \leq j \leq l. \\
M(i,j)=0 \text{ otherwise}. 
\end{array}
\right.
\end{align*}
\label{def-l3}
\end{definition}

\begin{proposition}
Any strictly upper triangular matrix of an $L_3$-formation is a CN-realizable matrix. 
\label{prop-loupe3}
\end{proposition}

\begin{proof}
Let $S^k_l=B^k_{k+1} B^k_{k+4} B^k_{k+5} \dots B^{k}_{l} B^{k+2}_{k+4} B^{k+2}_{k+5} \dots B^{k+2}_{l} B^{k+3}_{l} B^{k+3}_{l-1} \dots B^{k+3}_{k+4}$. 
Then $S^k_l$ is transformed into $W^{k}_{k+1} W^{k+3}_{k+4} W^{k+2}_{k+3} S^{k+1}_l W^{k+2}_{k+3} W^{k+3}_{k+4} W^{k}_{k+1}$ by the ladder moves as shown in Figure \ref{fig-lou3}. 
Let $M$ be an $n \times n$ matrix of an $L_3(k,l)$-formation. 
Take a B-ladder diagram $D$ of $M$ as $D=(B^k_{k+2}) S^k_l (B^{k+1}_{k+4})$. 
Apply the ladder moves until we obtain $S^{l-4}_l$, which is transformed into white edges as shown in Figure \ref{fig-lou3}.
If $M(k, k+2)$ or $M(k+1, k+4)$ is 2, use the white edge $W^{k}_{k+1}$ or the sequence $W^{k+1}_{k+2} W^{k+2}_{k+3} W^{k+3}_{k+4} W^{k}_{k+1}$ to shorten them. 
\end{proof}
\begin{figure}[ht]
\centering
\includegraphics[width=11cm]{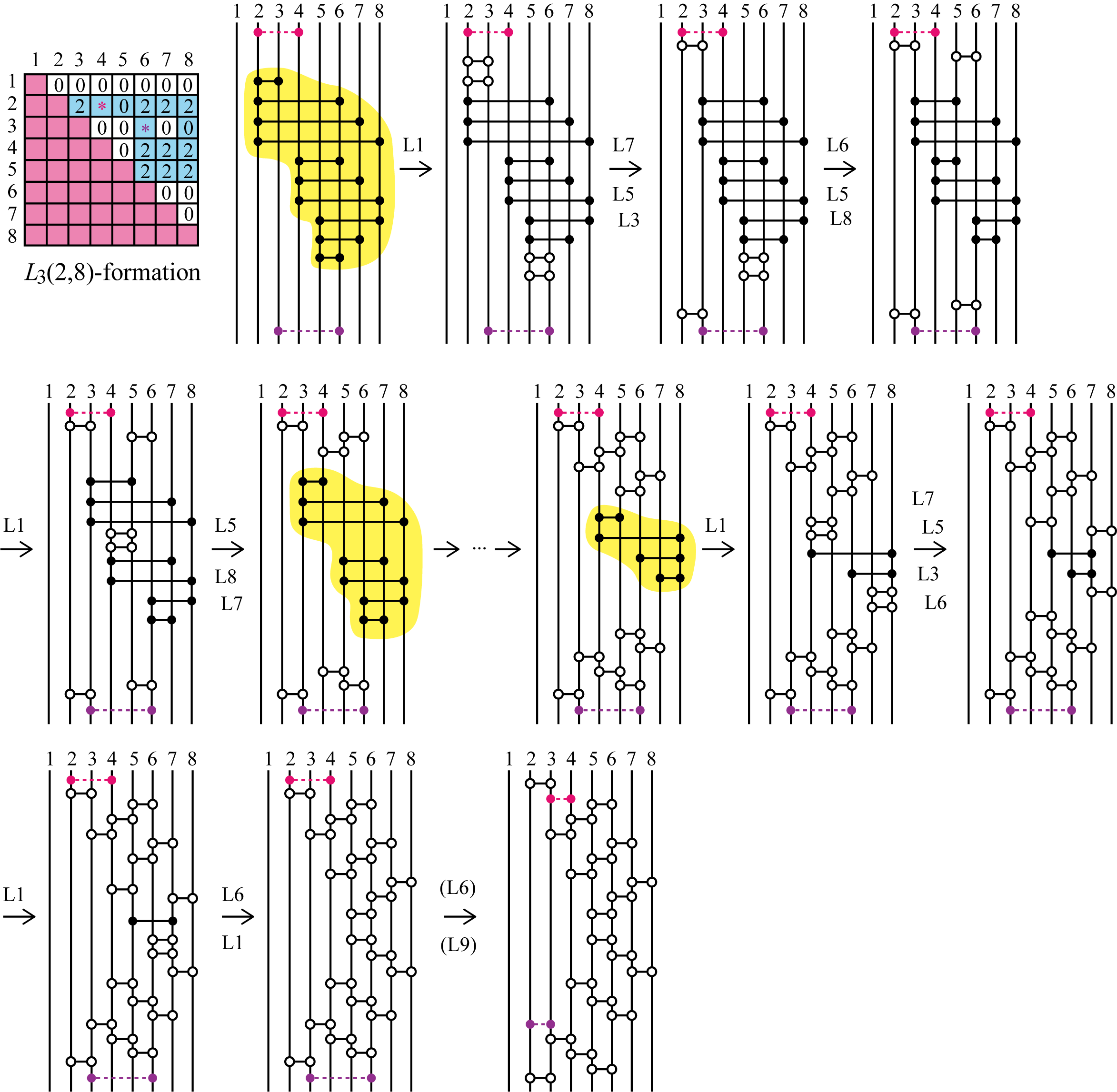}
\caption{A matrix of an $L_3(2,8)$-formation is a CN-realizable matrix. $S^2_8$, $S^3_8$, $S^4_8$ are highlighted in yellow. }
\label{fig-lou3}
\end{figure}

\begin{corollary}
Let $M$ be an $n \times n$ matrix of an $L_1 (k,l)$, $L_2(k,l)$ or $L_3(k,l)$-formation. 
Let $M_1$ be an $n\times n$ matrix such that $M_1(k,l)=2$ and the other entries are zero. 
Then, the matrix $M-M_1$ is a CN-realizable matrix.
\label{cor-l123}
\end{corollary}

\begin{proof}
Ignore the edges corresponding to $B^k_l$ in the proofs of Propositions \ref{prop-loupe}, \ref{prop-loupe2} and \ref{prop-loupe3}. 
\end{proof}

\noindent We call a formation of Definition \ref{def-l1}, \ref{def-l2}, \ref{def-l3} or Corollary \ref{cor-l1} or \ref{cor-l123} a {\it loupe formation}.

\subsection{T-structure}
\label{subsec-T}

In this subsection, we attempt to start a general discussion of the CN-realizable formation. 
We define and discuss the ``T-structure'' of a strictly upper triangular $(0,2)$-matrix. 

\begin{definition}
Let $M$ be an $n \times n$ strictly upper triangular $(0,2)$-matrix. 
Consider an $n \times n$ grid that has a vertex $V(i,j)$ in the square at the $i^{th}$ from the left-hand side and the $j^{th}$ from the top when $M(i,j)=2$ (see Figure \ref{fig-PO} (2)). 
We call the alignment of the vertices in the grid the {\it grid alignment of $M$} and denote it by $G(M)$. 
\end{definition}
\begin{figure}[ht]
\centering
\includegraphics[width=8cm]{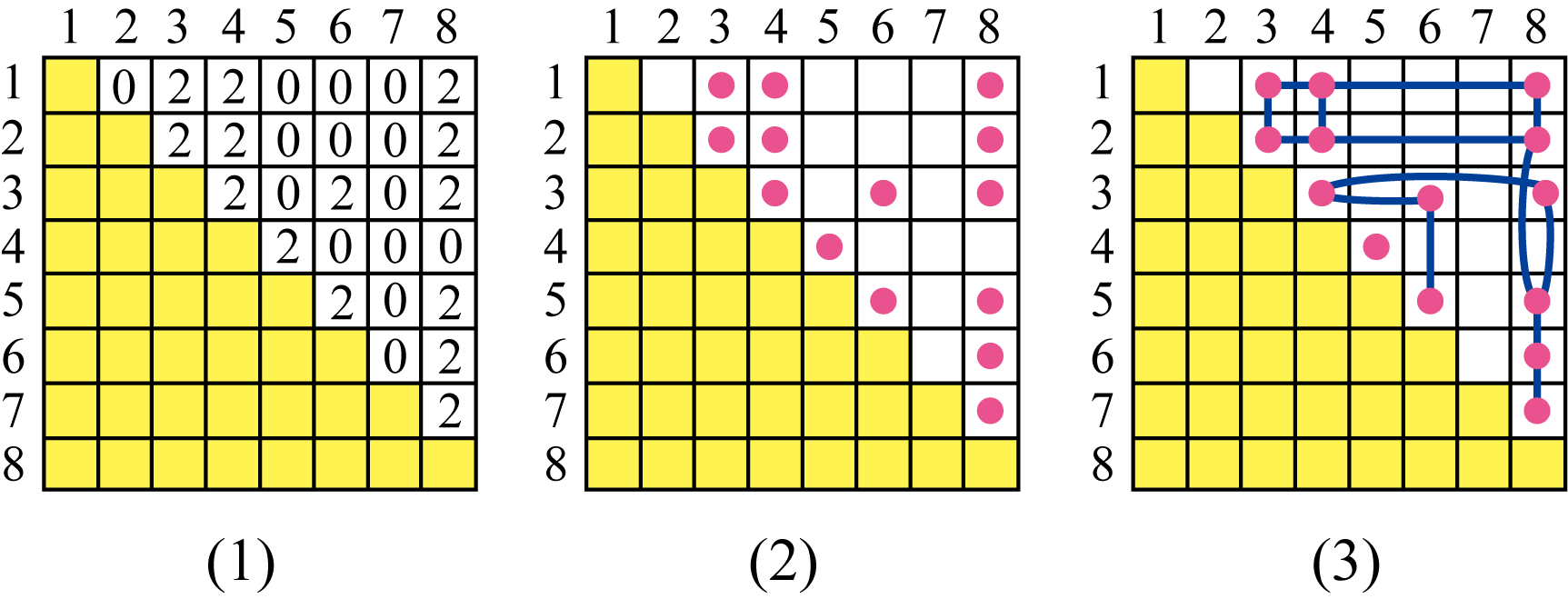}
\caption{(1): A $(0,2)$-matrix $M$. (2): The grid alignment $G(M)$ of $M$. (3): A graph $g$ on $G(M)$.}
\label{fig-PO}
\end{figure}

\begin{definition}
For a grid alignment $G(M)$, a {\it graph on $G(M)$} is a graph in the grid that is obtained from $G(M)$ by connecting some vertices by edges with the following three rules (see Figure \ref{fig-PO} (3)). 
\begin{itemize}
\item Each edge is vertical or horizontal; for each edge $E(V(i,j), V(k,l))$ connecting $V(i,j)$ and $V(k,l)$, $i=k$ or $j=l$. 
\item For each vertex $V(i,j)$, there is at most one vertex $V(i,k)$ on the left-hand side that shares an edge with $V(i,j)$; for each $V(i,j)$, the edge $E(V(i,j), V(i,k))$ with $k<j$ is unique if it exists. 
\item For each vertex $V(i,j)$, there is at most one vertex $V(l,j)$ below that shares an edge with $V(i,j)$; for each $V(i,j)$, the edge $E(V(i,j), V(l,j))$ with $i<l$ is unique if it exists. 
\end{itemize}
\label{def-graph-on}
\end{definition}

\begin{definition}
Let $g$ be a graph on a grid alignment $G(M)$. 
For a vertex $V(i,j)$, take the maximal sequence $V(i,j) \rightarrow V(i,j_1) \rightarrow V(i,j_2) \rightarrow \dots \rightarrow V(i,j_n)$ with $j=j_0>j_1 >j_2 > \dots > j_n$, where each pair of vertices $V(i, j_{l-1})$ and $V(i,j_l)$ share an edge. 
We call the sequence the {\it horizontal path for $V(i,j)$ with length $n$}. 
In the same way, take the maximal sequence $V(i,j) \rightarrow V(i_1,j) \rightarrow V(i_2,j) \rightarrow \dots \rightarrow V(i_m,j)$ with $i=i_0<i_1 <i_2 < \dots < i_m$, where each pair of vertices $V(i_{k-1}, j)$ and $V(i_k,j)$ share an edge. 
We call the sequence the {\it vertical path for $V(i,j)$ with length $m$}. 
\end{definition}

\noindent By the second and third conditions of Definition \ref{def-graph-on}, the vertical and horizontal paths are determined uniquely for each $V(i,j)$. 

\begin{example}
For the graph of Figure \ref{fig-PO} (3), the vertex $V(3,8)$ has the horizontal path $V(3,8) \rightarrow V(3,4)$ of length one and the vertical path $V(3,8) \rightarrow V(5,8) \rightarrow V(6,8) \rightarrow V(7,8)$ of length three. 
The lengths of the horizontal and vertical paths for $V(2,3)$ are zero. 
\end{example}

\begin{definition}
Let $g$ be a graph on a grid alignment $G(M)$. 
We say that $g$ has a {\it T-structure} if $g$ satisfies the following three conditions. 
\begin{itemize}
\item[(C1)] For each vertex $V(i,j)$ of $g$, let $m$, $n$ be the length of the vertical, horizontal paths for $V(i,j)$, respectively. Then $m+n=j-i-1$.
\item[(C2)] If there are three edges $E(V(i,k), V(i,j))$, $E(V(i,j), V(l,j))$, $E(V(i,k), V(m,k))$ ($k<j$, $i<l$, $i<m$), then $m=l$ and there also is the edge $E(V(l,k), V(l,j))$. 
\item[(C3)] If there are three edges $E(V(i,k), V(i,j))$, $E(V(i,j), V(l,j))$, $E(V(l,j), V(l,m))$ ($k<j$, $i<l$, $m<j$), then $m=k$ and there also is the edge $E(V(l,k), V(i,k))$. 
\end{itemize}
\begin{figure}[ht]
\centering
\includegraphics[width=10cm]{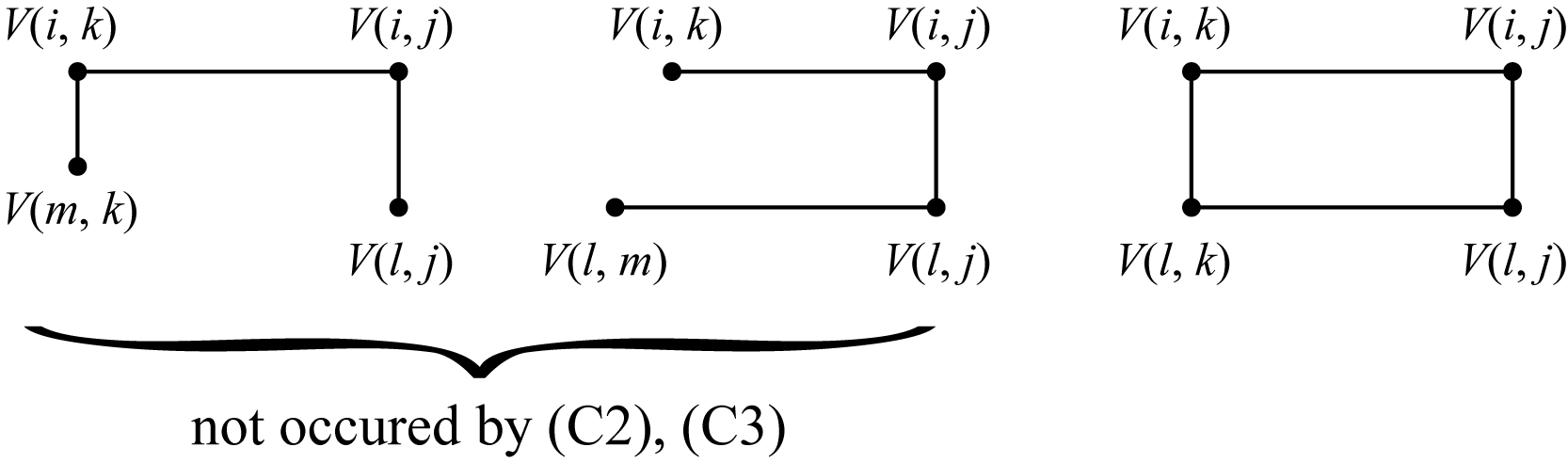}
\caption{On (C2) and (C3). }
\label{fig-C5}
\end{figure}
(See Figure \ref{fig-C5} for conditions C2, C3.) 
We say that a matrix $M$ has a T-structure if the grid $G(M)$ admits a graph on $G(M)$ that has a T-structure. 
\label{def-C1C2}
\end{definition}

\begin{example}
The graph in Figure \ref{fig-PO} (3) has a T-structure. 
The graph in Figure \ref{fig-PO-ex} fails to have a T-structure; 
the vertex $V(3,8)$ does not meet the condition (C1). 
The graph also has the undesired parts that are mentioned in the conditions (C2) and (C3). 
\begin{figure}[ht]
\centering
\includegraphics[width=2.5cm]{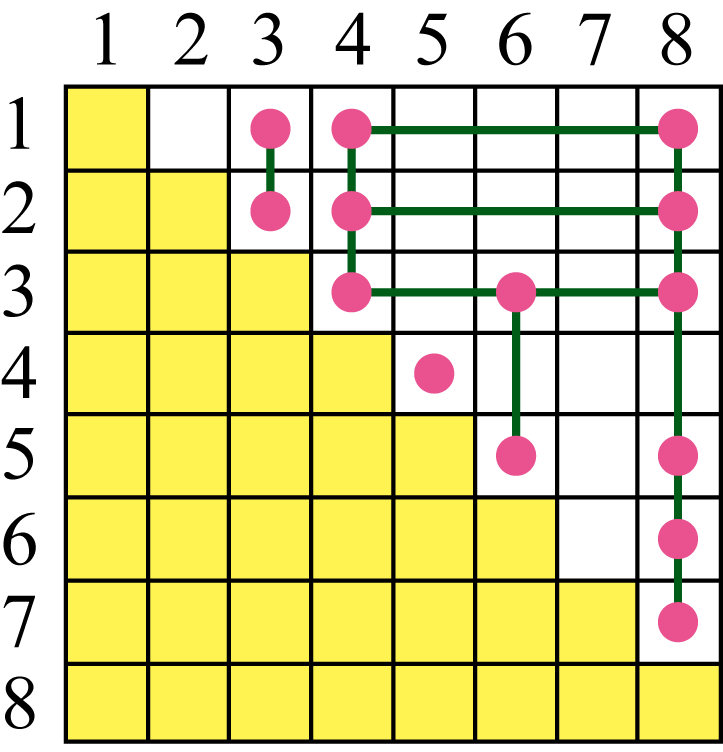}
\caption{A graph which fails to have a T-structure. }
\label{fig-PO-ex}
\end{figure}
\end{example}

\noindent By definition, all the matrices of a snake, hang-glider or loupe formation have a T-structure. 
We raise the following conjecture. 

\begin{conjecture}
Let $M$ be a strictly upper triangular $(0,2)$-matrix such that there is at most one pair $(i,j)$ that satisfies $j-i=k$ and $M(i,j)=2$ for each $k \geq 3$. 
If $M$ has a T-structure, then $M$ is a CN-realizable matrix\footnote{There also exsist CN-realizable matrices with no T-structure. 
For example, see the matrix $M$ in Figure \ref{fig-l-ex}. 
The vertex $V(1,4)$ of the grid alignment $G(M)$ needs to have a vertical path of length two since there are no vertices on the left-hand side. 
The vertex $V(2,6)$ needs to have the horizontal path of length two and the vertical path of length one. 
Then, the undesired part consisting of the vertices $V(3,4)$, $V(2,4)$, $V(2,6)$, and $V(5,6)$ is unavoidable, which is against the condition (C2). 
Hence, no T-structures are allowed for $M$.}. 
\end{conjecture}

\begin{example}
The matrix $M$ shown in Figure \ref{fig-alg-ex} has a T-structure and we can transform a B-ladder diagram of $M$ into a W-ladder diagram by ladder moves\footnote{Except for L4.} based on the T-structure in the following way. 
Take a B-ladder diagram of $M$ so that shorter black edges are lower than the longer edges, where the length of a black edge $B^k_l$ means the value $l-k$. 
Apply the ladder move L1 to the black edges of length one to convert them to white edges (the transformation $A$ in Figure \ref{fig-alg-ex}). 
Next, for the black edges $B^i_j$ of length 2, use the white edegs such that the corresponding vertices are on the vertical or horizontal path for $V(i,j)$ to shorten them (for example, use $W^1_2$ for $B^1_3$ in Figure \ref{fig-alg-ex}) and apply L1 to convert them to white edges (the transformation $B$ in Figure \ref{fig-alg-ex}). 
Repeat the procedure until the longest black edge is shortened and converted to white edges. 
\begin{figure}[ht]
\centering
\includegraphics[width=10cm]{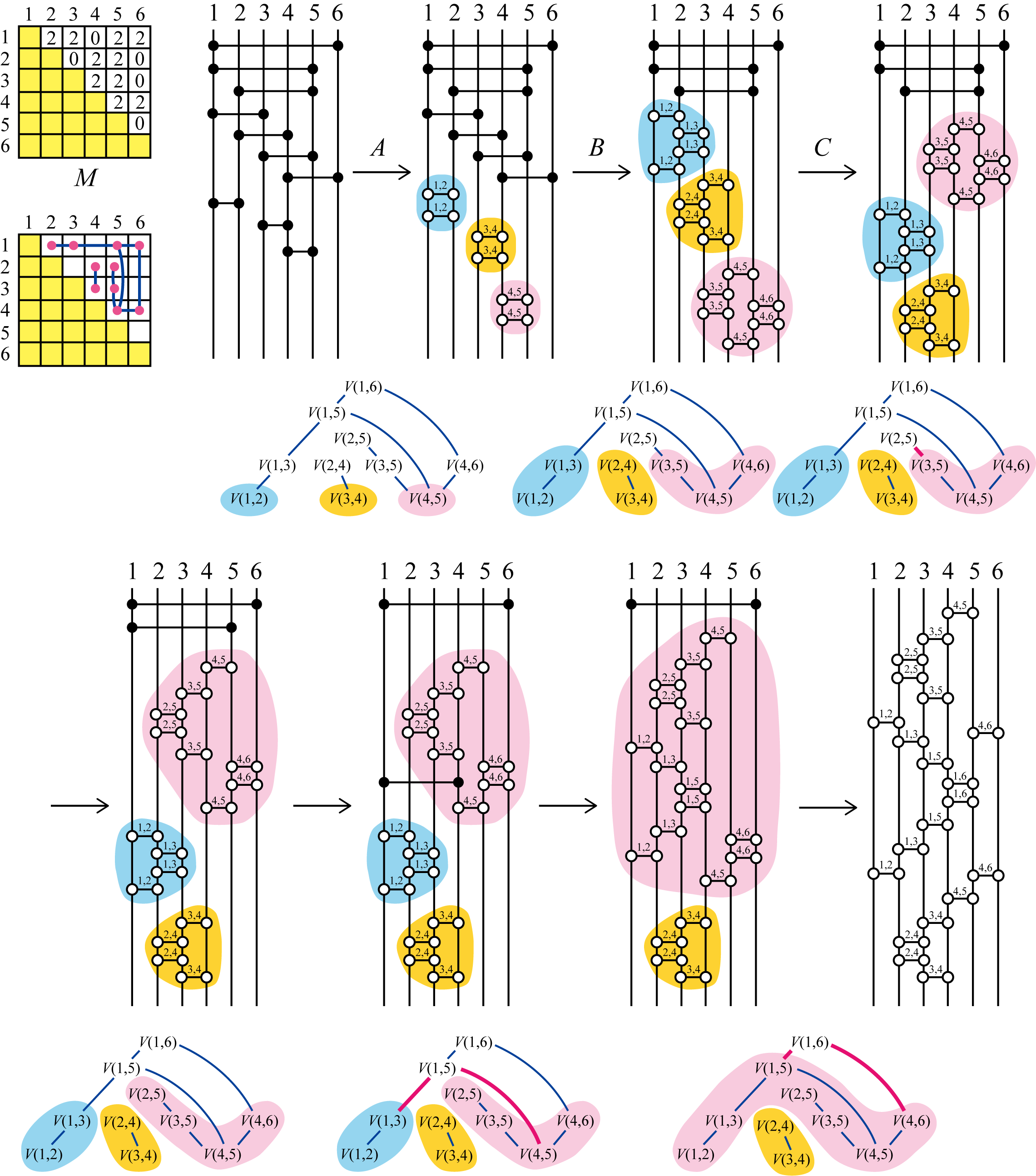}
\caption{The matrix $M$ is CN-realizable. For the transformation $C$, convert all the white edges highlighted in pink to black edges $B^3_5 B^4_6 B^4_5$ once, move them above the yellow- and blue-colored edges by ladder moves, and convert them again to the white edges. }
\label{fig-alg-ex}
\end{figure}
\end{example}

\section{CN-realizable configurations}
\label{section-conf}

In this section, we discuss the CN-realizable configurations. 

\begin{definition}
We say that a configuration $c$ of an $n \times n$ matrix is a {\it CN-realizable configuration} if any $n \times n$ matrix that has $c$ is CN-realizable. 
\end{definition}

\noindent By Proposition \ref{prop-reverse}, we have the following proposition. 

\begin{proposition}[\cite{AY-5}]
The reverse of a CN-realizable configuration is also a CN-realizable configuration. 
\end{proposition}

\begin{example}
The configurations $a1$ to $a4$ shown in Figure \ref{fig-m-a} are CN-realizable configurations for a $6 \times 6$ T0 upper triangular $(0,2)$-matrix. 
\begin{figure}[ht]
\centering
\includegraphics[width=6cm]{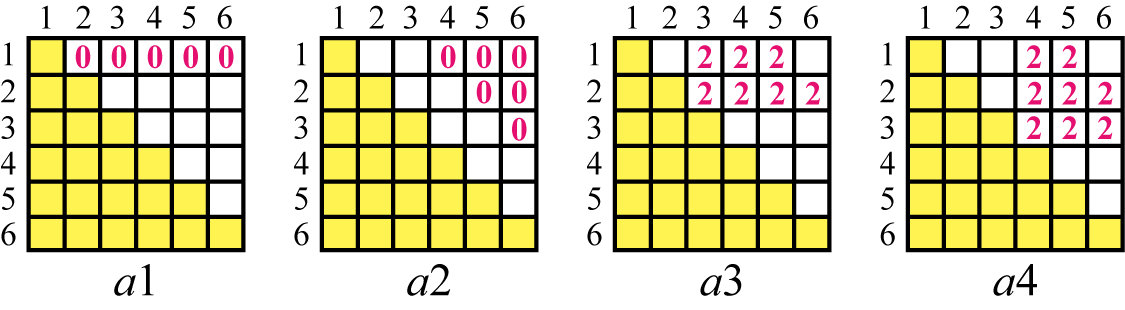}
\caption{The configurations $a1$ to $a4$. }
\label{fig-m-a}
\end{figure}
\label{ex-2}
\end{example}

\begin{proof}
For $a1$, it follows from Lemma \ref{lem-15}. 
For $a2$, it follows from Proposition \ref{prop-3zero}. 
For $a3$ and $a4$, it follows from Lemma \ref{lem-15}, Propositions \ref{prop-02-sum} and Figure \ref{fig-g}. 
For example, let $M$ be a T0 upper triangular $(0,2)$-matrix that has the configuration $a3$. 
Divide $M$ into three matrices as 
\begin{align*}
M=
\begin{bmatrix}
0 & * & 0 & 0 & 0 & 0 \\
0 & 0 & 0 & 0 & 0 & 0 \\
0 & 0 & 0 & 0 & 0 & 0 \\
0 & 0 & 0 & 0 & 0 & 0 \\
0 & 0 & 0 & 0 & 0 & 0 \\
0 & 0 & 0 & 0 & 0 & 0 \\
\end{bmatrix}
+
\begin{bmatrix}
0 & 0 & 0 & 0 & 0 & 0 \\
0 & 0 & 0 & 0 & 0 & 0 \\
0 & 0 & 0 & * & * & * \\
0 & 0 & 0 & 0 & * & * \\
0 & 0 & 0 & 0 & 0 & * \\
0 & 0 & 0 & 0 & 0 & 0 \\
\end{bmatrix}
+
\begin{bmatrix}
0 & 0 & 2 & 2 & 2 & * \\
0 & 0 & 2 & 2 & 2 & 2 \\
0 & 0 & 0 & 0 & 0 & 0 \\
0 & 0 & 0 & 0 & 0 & 0 \\
0 & 0 & 0 & 0 & 0 & 0 \\
0 & 0 & 0 & 0 & 0 & 0 \\
\end{bmatrix},
\end{align*}
where $*$ is either 0 or 2. 
The first and second matrices are T0 by the assumption, and CN-realizable by Lemma \ref{lem-15}. 
The third is CN-realizable by Figure \ref{fig-g}. 
Hence, $M$ is also CN-realizable by Proposition \ref{prop-02-sum}. 
\begin{figure}[ht]
\centering
\includegraphics[width=6cm]{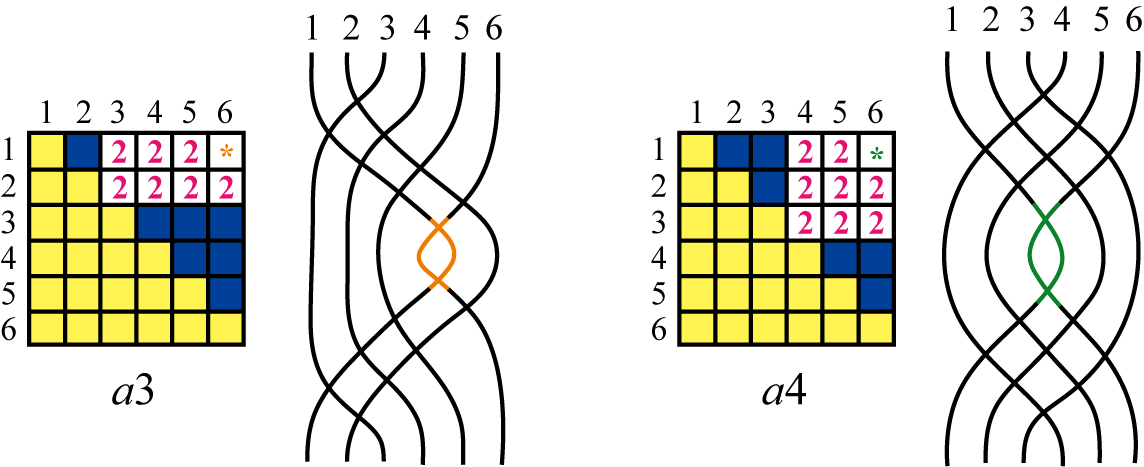}
\caption{The confugurations $a3$ and $a4$ are CN-realizable configurations. 
Unhook the $1^{st}$ and $6^{th}$ strands if the $(1,6)$ entry is 0.}
\label{fig-g}
\end{figure}
\end{proof}

\noindent The formations shown in Section \ref{section-formation} are useful for finding further CN-realizable configurations. 

\begin{example}
The configurations $b1, b2, b3$ shown in Figure \ref{fig-i} are CN-realizable configurations for a $6 \times 6$ T0 upper triangular $(0,2)$-matrix. 
\begin{figure}[ht]
\centering
\includegraphics[width=5cm]{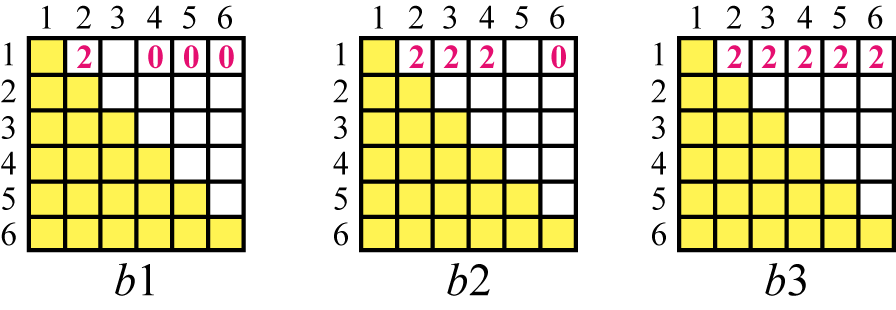}
\caption{The configurations $b1, b2, b3$. }
\label{fig-i}
\end{figure}
\label{ex-1}
\end{example}

\begin{proof}
It follows from Lemma \ref{lem-15}, Propositions \ref{prop-02-sum}, and \ref{prop-r-form}. 
For example, let $M$ be a T0 upper triangular $(0,2)$-matrix that has the configuration $b2$. 
Divide $M$ into two matrices as 
\begin{align*}
M=
\begin{bmatrix}
0 & 2 & 2 & 2 & * & 0 \\
0 & 0 & 0 & 0 & 0 & 0 \\
0 & 0 & 0 & 0 & 0 & 0 \\
0 & 0 & 0 & 0 & 0 & 0 \\
0 & 0 & 0 & 0 & 0 & 0 \\
0 & 0 & 0 & 0 & 0 & 0 \\
\end{bmatrix}
+
\begin{bmatrix}
0 & 0 & 0 & 0 & 0 & 0 \\
0 & 0 & * & * & * & * \\
0 & 0 & 0 & * & * & * \\
0 & 0 & 0 & 0 & * & * \\
0 & 0 & 0 & 0 & 0 & * \\
0 & 0 & 0 & 0 & 0 & 0 \\
\end{bmatrix}.
\end{align*}
The first matrix is CN-realizable by Proposition \ref{prop-r-form} as it forms the $r(1,k)$-formation for $k=4$ or 5. 
The second matrix is T0 by assumption, and CN-realizable by Lemma \ref{lem-15}. 
Hence, $M$ is also CN-realizable by Proposition \ref{prop-02-sum}. 
\end{proof}

\noindent In the same way, we obtain CN-realizable configurations derived from the snake, hang-glider and loupe formations as follows. 

\begin{proposition}
The configurations $c1$ to $c100$ shown in Figures \ref{fig-table1}, \ref{fig-table2}, \ref{fig-table3}, \ref{fig-table4}, \ref{fig-table5} are CN-realizable configurations for a $6 \times 6$ T0 upper triangular $(0,2)$-matrix.
\label{prop-CN-form}
\end{proposition}

\noindent We also obtain CN-realizable configurations using the ladder diagram. 

\begin{proposition}
The configurations $d1$ to $d21$ shown in Figures \ref{fig-d1} and \ref{fig-d2} are CN-realizable configurations for a $6 \times 6$ T0 upper triangular $(0,2)$-matrix. 
\label{prop-CN-ladder}
\end{proposition}

\section{Proof of Theorem \ref{thm-true6}}
\label{section-66}

We discuss Conjecture \ref{conj-CN} for $n \leq 6$ in Section \ref{sub-sec-66}, Conjecture \ref{conj-OU} in Section \ref{sub-sec-OU} and Conjecture \ref{conj-C} in Section \ref{sub-sec-CM} to prove Theorem \ref{thm-true6}.

\subsection{CN-realizable $6 \times 6$ matrix}
\label{sub-sec-66}

In this subsection, we give a positive answer to Conjecture \ref{conj-CN} when $n \leq 6$.

\begin{proposition}
A $6 \times 6$ upper triangular $(0,2)$-matrix $M$ is CN-realizable if and only if $M$ is T0. 
\label{prop-02-CN6}
\end{proposition}

\begin{proof}
If an upper triangular $(0,2)$-matrix $M$ is CN-realizable, then $M$ is T0 by Proposition \ref{prop-T0}. 
All the $6 \times 6$ upper triangular T0 $(0,2)$-matrices, whose number is 4,824, have at least one of the CN-realizable configurations or their reverses of Examples \ref{ex-2}, \ref{ex-1} or Propositions \ref{prop-CN-form}, \ref{prop-CN-ladder}. 
It was confirmed by eliminating the matrices that have the CN-realizable configurations from the list of the 4,824 T0 matrices by computer\footnote{In the previous version of the paper, another set of the CN-realizable configurations was used to prove the proposition. See \url{https://arxiv.org/pdf/2506.08659v1}.}. 
\end{proof}

\begin{remark}
The number of the upper triangular T0 $(0,2)$-matrices was calcurated by the computer, and the following table shows the numbers up to $7 \times 7$ matrices.  \\

\begin{center}
\begin{tabular}{ccccccc}\hline
1 & 2 & 3 & 4 & 5 & 6 & 7 \\ \hline
1 & 2 & 8 & 64 & 1,024 & 32,768 & 2,097,152 \\
1 & 2 & 7 & 40 & 357 & 4,824 & 96,428 \\ \hline
\end{tabular}
\end{center}

\noindent The first row indicates the size $n$ of matrix, the second row indicates the number of zero-diagonal symmetric $(0,2)$-matrices, which is $2^{n(n-1)/2}$, and the third row indicates the number of zero-diagonal symmetric T0 $(0,2)$-matrices. 
\end{remark}

\begin{proposition}
A non-negative integer $6 \times 6$ matrix $M$ is the CN matrix of some pure $6$-braid projection if and only if $M$ is an even symmetric T0 matrix. 
\label{prop-CN6}
\end{proposition}

\begin{proof}
If a matrix $M$ is CN-realizable by a pure braid projection, then $M$ is symmetric by definition and $M$ is a non-negative even T0 matrix by Propositions \ref{prop-pure-even} and \ref{prop-T0}. 
By Propositions \ref{prop-M02} and \ref{prop-02-CN6}, all the $6 \times 6$ non-negative even T0 symmetric matrix are CN-realizable by a pure braid projection. 
\end{proof}

\subsection{OU matrix}
\label{sub-sec-OU}

In this subsection, we give a positive answer to Conjecture \ref{conj-OU} when $n \leq 6$. 
The following proposition was shown in \cite{AY-5} for $N=5$. 

\begin{proposition}
If Conjecture \ref{conj-CN} is true for $n \leq N$, then Conjecture \ref{conj-OU} is true for $n \leq N$. 
\label{prop-3to2}
\end{proposition}

\begin{proof}
Let $M$ be an $n\times n$ non-negative matrix. 
Suppose that $M$ is the OU matrix of a pure $n$-braid diagram $b$. 
By the definitions of the CN matrix and the OU matrix, $M+M^T$ is the CN matrix of the projection of $b$. 
By Proposition \ref{prop-T0}, $M+M^T$ is an even T0 matrix. 
Conversely, suppose that $M+M^T$ is an even T0 matrix. 
By the assumption that Conjecture \ref{conj-CN} is true for $n\leq N$, $M+M^T$ is the CN matrix of some pure $n$-braid projection $B$ for $n\leq N$. 
By Corollary 1 in \cite{AY-5}, $M$ is the OU matrix of some pure $n$-braid diagram whose projection is $B$ for $n\leq N$. 
This completes the proof of Proposition \ref{prop-3to2}.
\end{proof}

\noindent From Propositions \ref{prop-CN6} and \ref{prop-3to2}, we have the following corollary, which is an answer to Conjecture \ref{conj-OU} for $n \leq 6$. 

\begin{corollary}
A $6 \times 6$ non-negative integer matrix $M$ is the OU matrix of some pure $6$-braid diagram if and only if $M+M^T$ is an even T0 matrix. 
\label{cor-OU6}
\end{corollary}

\subsection{Crossing matrix}
\label{sub-sec-CM}

In this subsection, we give a positive answer to Conjecture \ref{conj-C} when $n \leq 6$ and prove Theorem \ref{thm-true6}. 
The following proposition was shown in \cite{AY-5} for $N=5$. 

\begin{proposition}
If Conjecture \ref{conj-CN} is true for $n \leq N$, then Conjecture \ref{conj-C} is true for $n \leq N$. 
\label{prop-3to1}
\end{proposition}

\begin{proof}
Let $M=(M(i,j))$ be an $n\times n$ matrix. 
Assume that Conjecture \ref{conj-CN} is true for $n\leq N$. 
Suppose that $M$ is the crossing matrix of a positive pure $n$-braid diagram $b$. 
Since $b$ is positive, $M$ coincides with the OU matrix of $b$ by Proposition \ref{prop-C-OU}, and $M$ is a non-negative integer matrix. 
Since $b$ is pure, $M$ is symmetric by Proposition \ref{prop-T0-set}. 
By Proposition \ref{prop-3to2}, $M+M^T=2M$ is an even T0 matrix for $n\leq N$, and hence $M$ is a T0 symmetric matrix for $n\leq N$. 
Conversely, suppose that $M$ is a non-negative T0 symmetric matrix. 
Then, $M+M^T=2M$ is a non-negative even T0 (symmetric) matrix. 
By the assumption that Conjecture 3 is true for $n\leq N$, $2M$ is the CN matrix of some pure $n$-braid projection $B$ for $n\leq N$. 
The number of intersections with the strands $s_i$ and $s_j$ of $B$ is $2M(i,j)$. 
Let $B(i,j)$ be the number of intersections between the strands $s_i$ and $s_j$ of $B$ where $s_i$ crosses $s_j$ from left to right. 
(We remark that $B(i,j)+B(j,i)=2M(i,j)$.) 
Since $B$ is pure, $B(i,j)=B(j,i)=M(i,j)$. 
Let $b$ be a positive pure $n$-braid diagram whose projection is $B$. 
Then, for each crossing $c$ between the strands $s_i$ and $s_j$ of $b$, $s_i$ crosses $s_j$ from left to right at $c$ if and only if $s_i$ crosses under $s_j$ at $c$. 
Hence, the number of crossings between $s_i$ and $s_j$ of $b$ where $s_i$ crosses over $s_j$ coincides with the number of crossings between $s_i$ and $s_j$ of $b$ where $s_i$ crosses under $s_j$. 
Then, the $(i,j)$ entry of the OU matrix of $b$ is $B(i,j) \ (=M(i,j))$, i.e., the OU matrix of $b$ is $M$. Since $b$ is positive, the crossing matrix of $b$ is also $M$ by Proposition \ref{prop-C-OU}. 
This completes the proof of Proposition \ref{prop-3to1}.
\end{proof}

\noindent From Propositions \ref{prop-CN6} and \ref{prop-3to1}, we have the following corollary, which is an answer to Conjecture \ref{conj-C} for $n \leq 6$. 

\begin{corollary}
A $6 \times 6$ integer matrix $M$ is the crossing matrix of some positive pure $6$-braid if and only if $M$ is a non-negative integer T0 symmetric matrix. 
\label{cor-CM6}
\end{corollary}

\medskip
\noindent We prove Theorem \ref{thm-true6}. \\

\noindent {\it Proof of Theorem \ref{thm-true6}.} \ It follows from Proposition \ref{prop-CN6}, Corollaries \ref{cor-OU6} and \ref{cor-CM6}. \qed

\section*{Acknowledgment}
The work of the second author was partially supported by the JSPS KAKENHI Grant Number JP21K03263. 
The work of the third author was partially supported by the JSPS KAKENHI Grant Number JP19K03508.

\begin{figure}[ht]
\centering
\includegraphics[width=13cm]{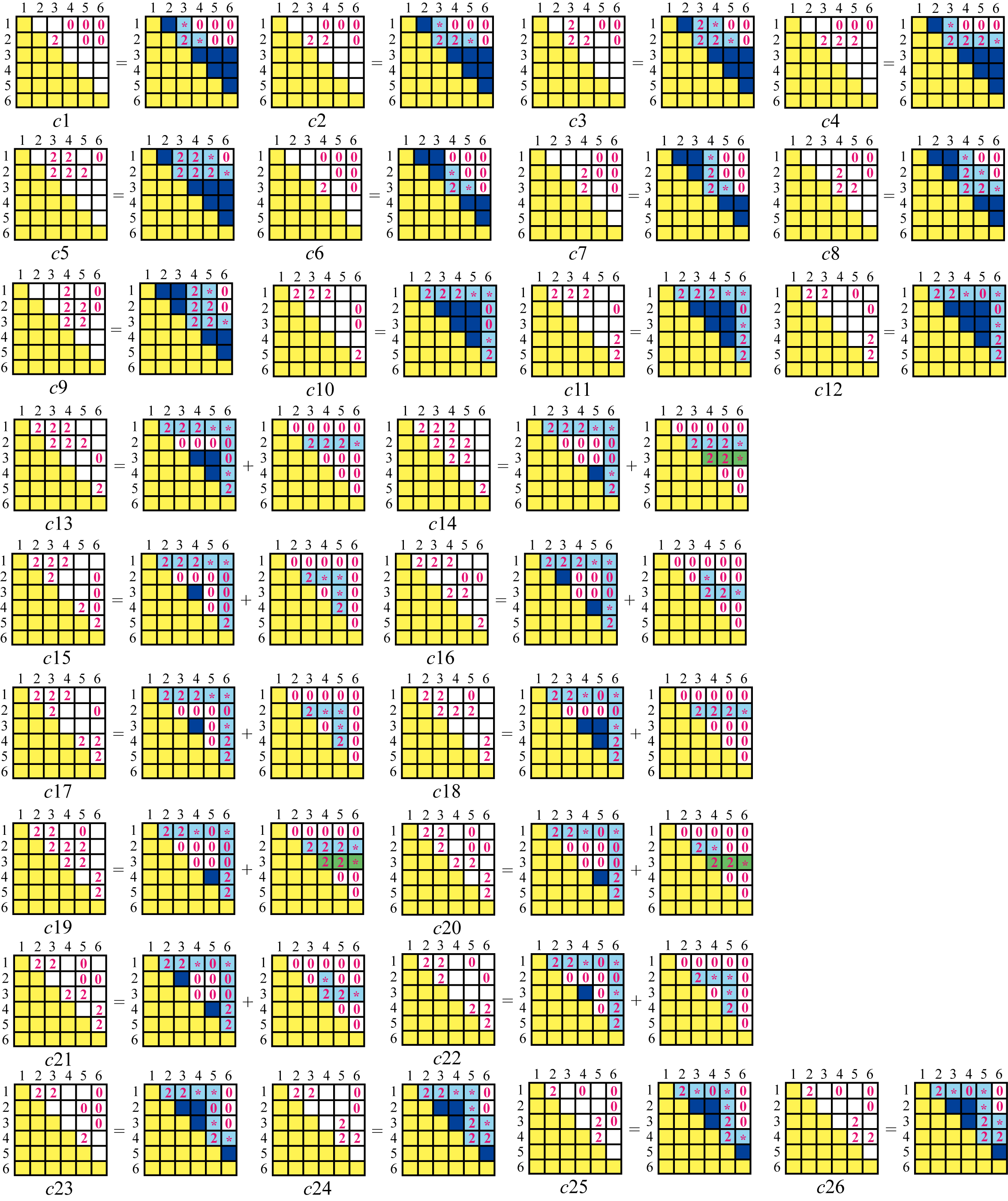}
\caption{The CN-realizable configurations $c1$ to $c26$ that are derived from the snake formation.}
\label{fig-table1}
\end{figure}

\begin{figure}[ht]
\centering
\includegraphics[width=13cm]{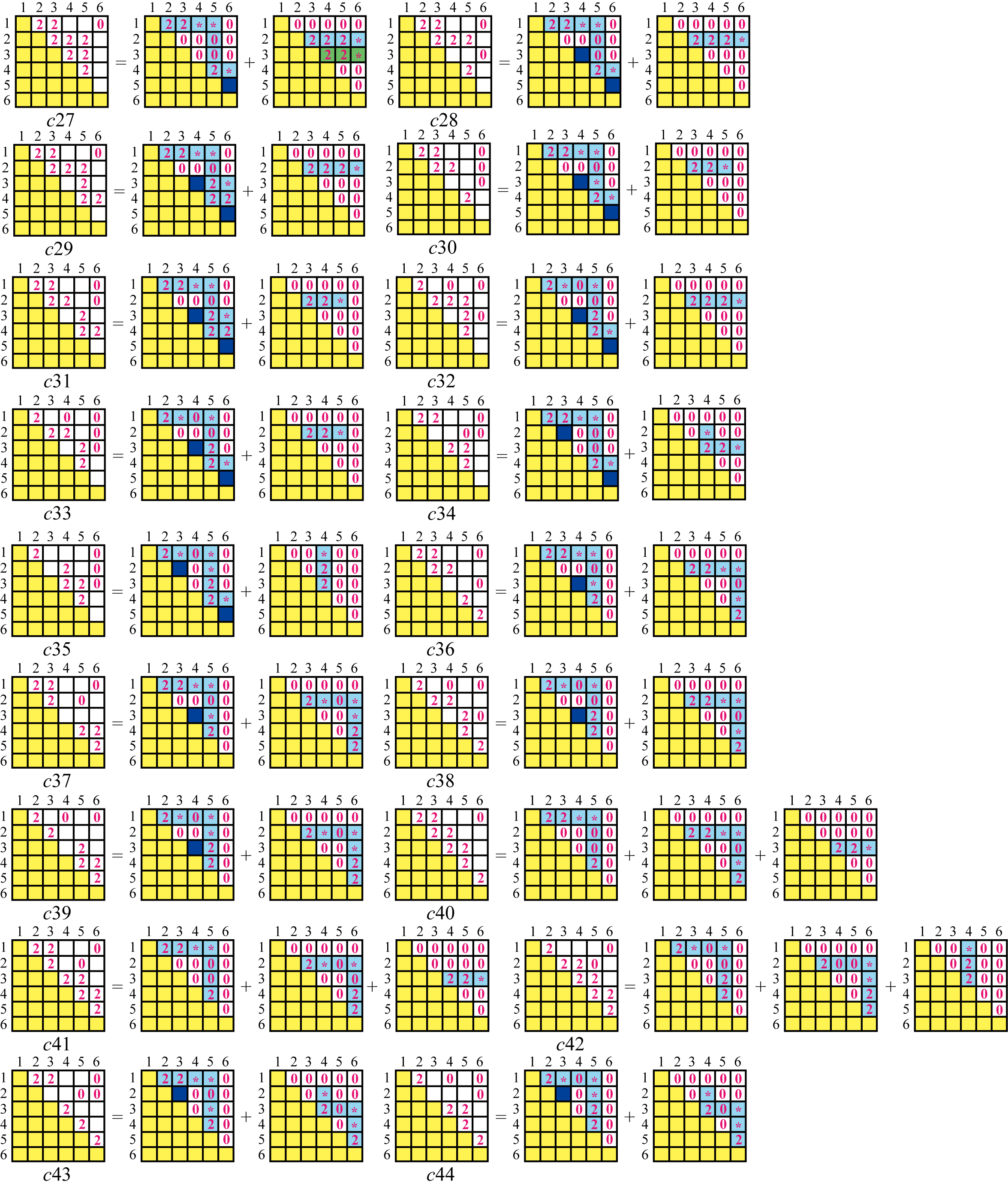}
\caption{The CN-realizable configurations $c27$ to $c44$ that are derived from the snake formation.}
\label{fig-table2}
\end{figure}

\begin{figure}[ht]
\centering
\includegraphics[width=13cm]{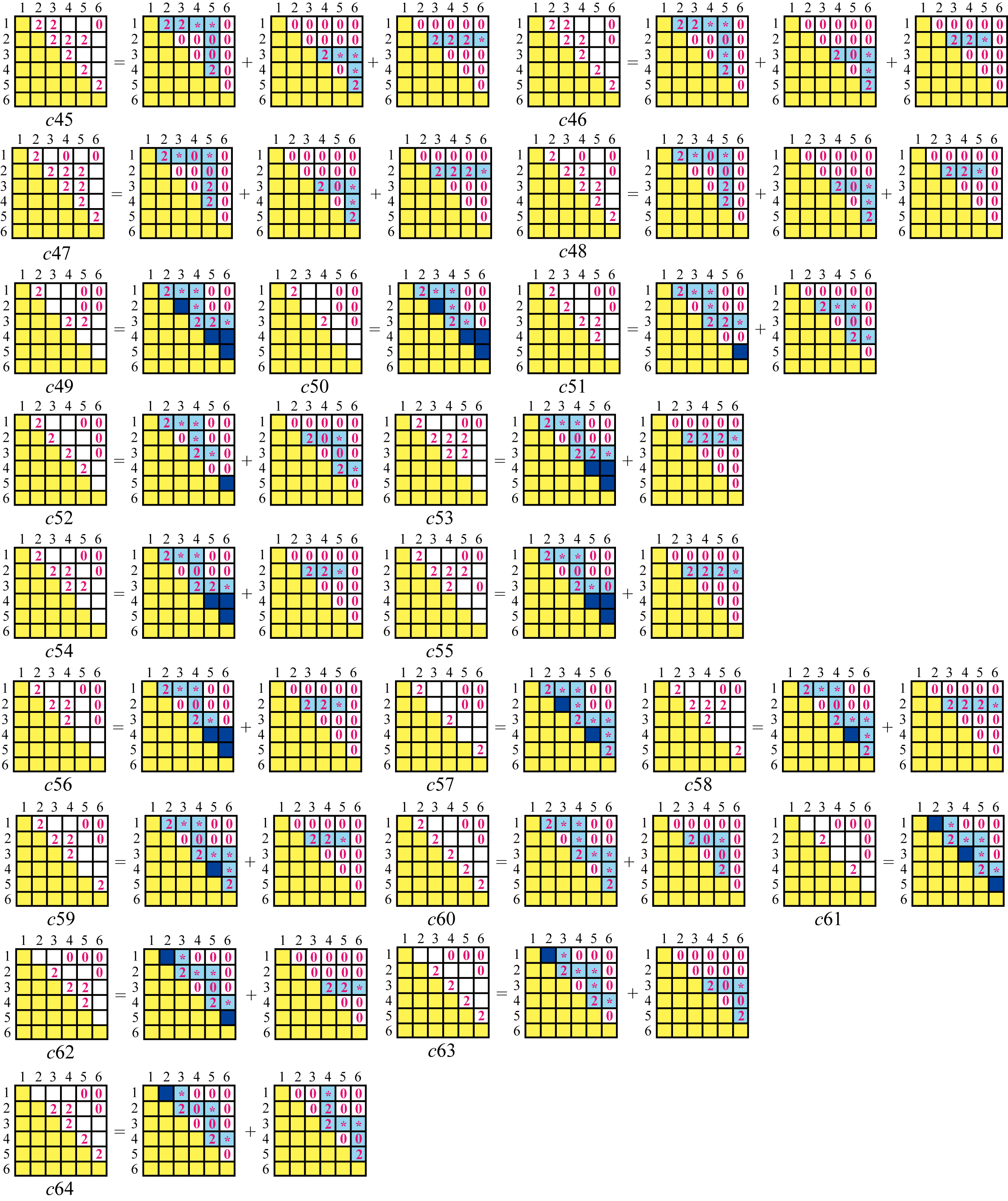}
\caption{The CN-realizable configurations $c45$ to $c64$ that are derived from the snake formation.}
\label{fig-table3}
\end{figure}

\begin{figure}[ht]
\centering
\includegraphics[width=13cm]{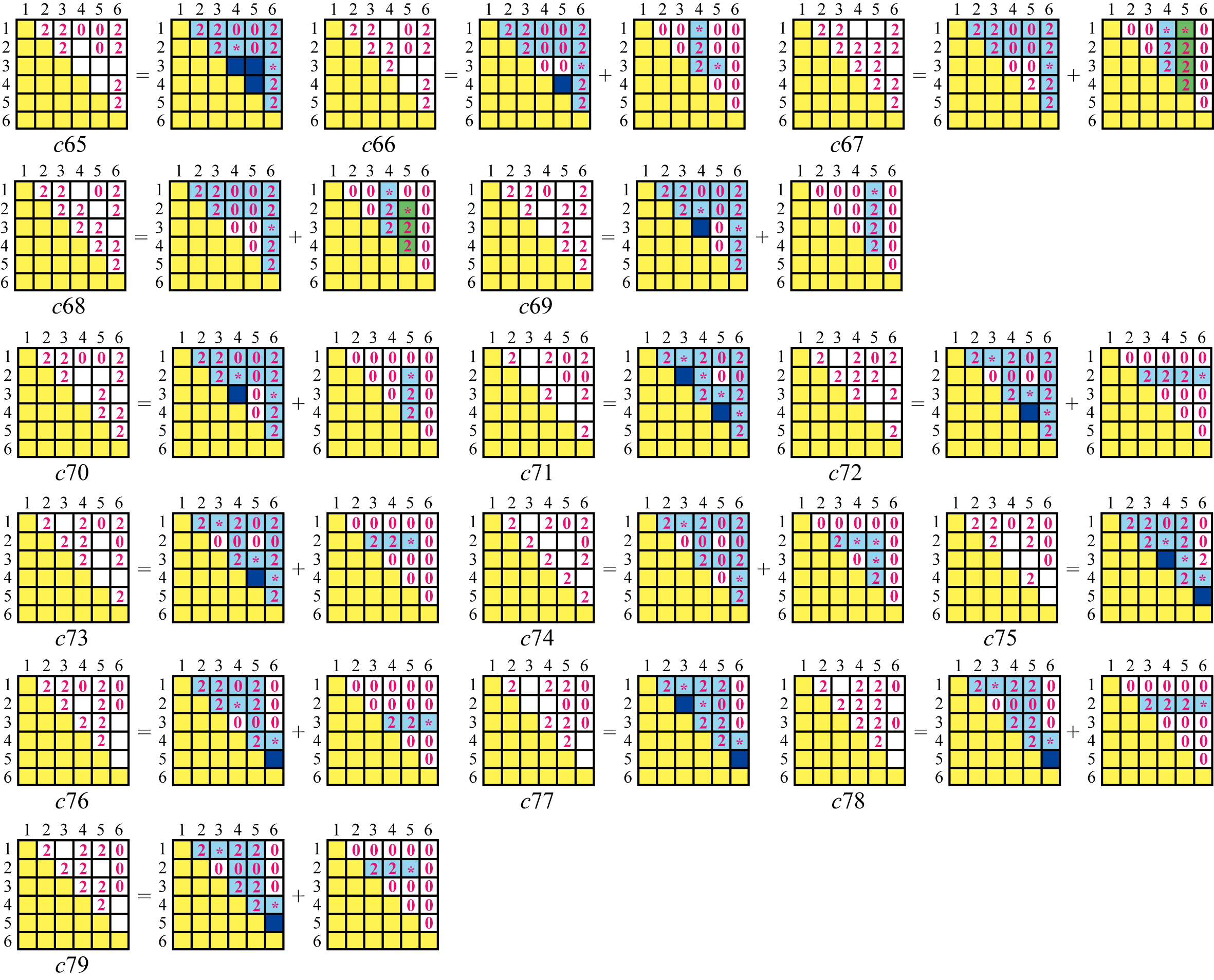}
\caption{The CN-realizable configurations $c65$ to $c79$ that are derived from the hang-glider and snake formation. 
The configurations $c65$ to $c70$ include the $H(1,6;2)$-fromation. 
The configurations $c71$ to $c74$ include the $H(1,6;3)$-formation. 
The configurations $c75$, $c76$ include the $H(1,5;2)$-formation. 
The configurations $c77$ to $c79$ include the $H(1,5;3)$-formation. }
\label{fig-table4}
\end{figure}

\begin{figure}[ht]
\centering
\includegraphics[width=13cm]{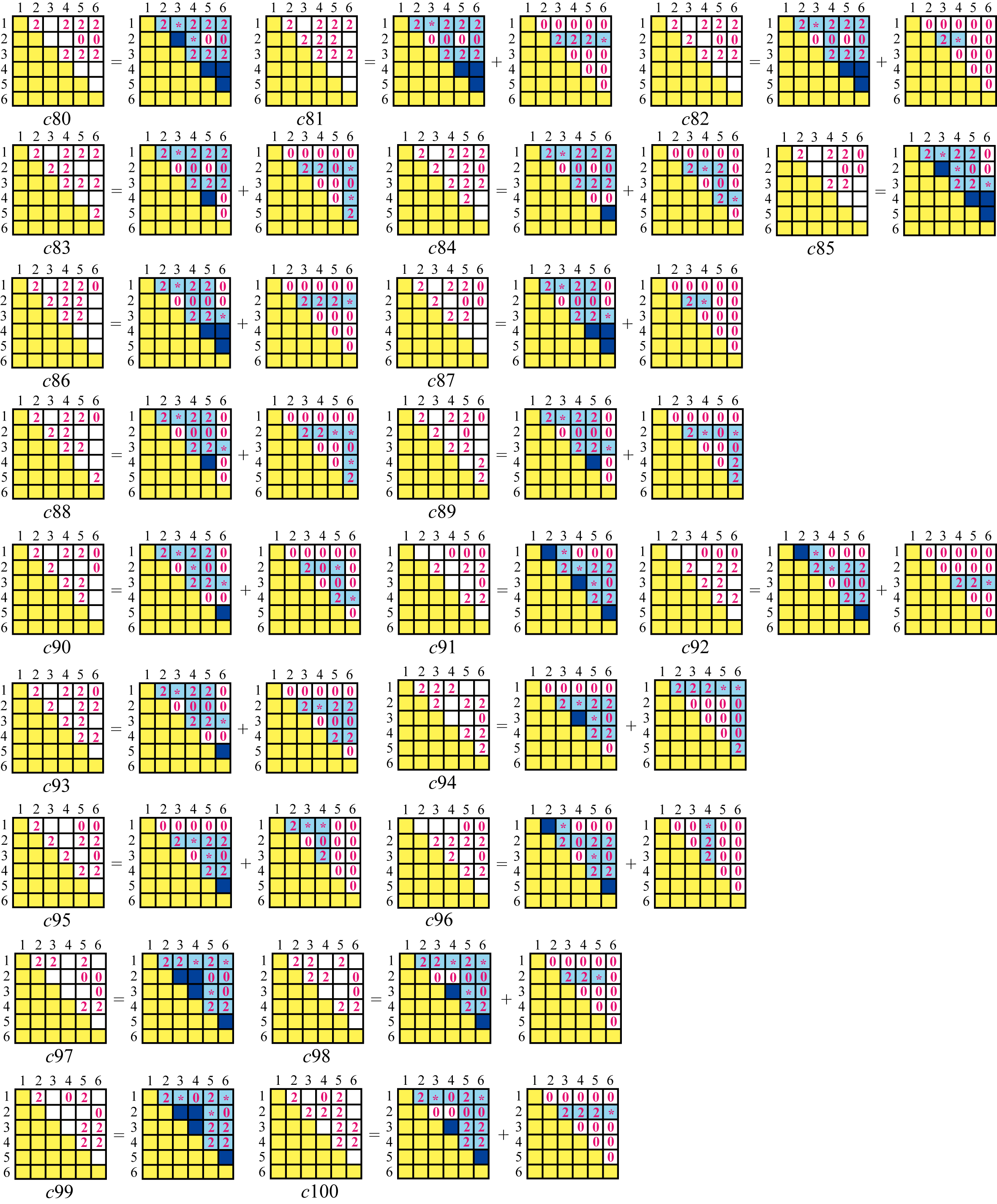}
\caption{The CN-realizable configurations $c80$ to $c100$ that are derived from the loupe and snake formations. 
The configurations $c80$ to $c84$ include the $L_1(1,6)$-fromation. 
The configurations $c85$ to $c90$ include the $L_1(1,5)$-formation. 
The configurations $c91$ to $c96$ include the $L_1(2,6)$-formation. 
The configurations $c97$, $c98$ include the $L_2(1,6)$-formation. 
The configurations $c99$, $c100$ include the $L_3(1,6)$-formation. }
\label{fig-table5}
\end{figure}

\begin{figure}[ht]
\centering
\includegraphics[width=13cm]{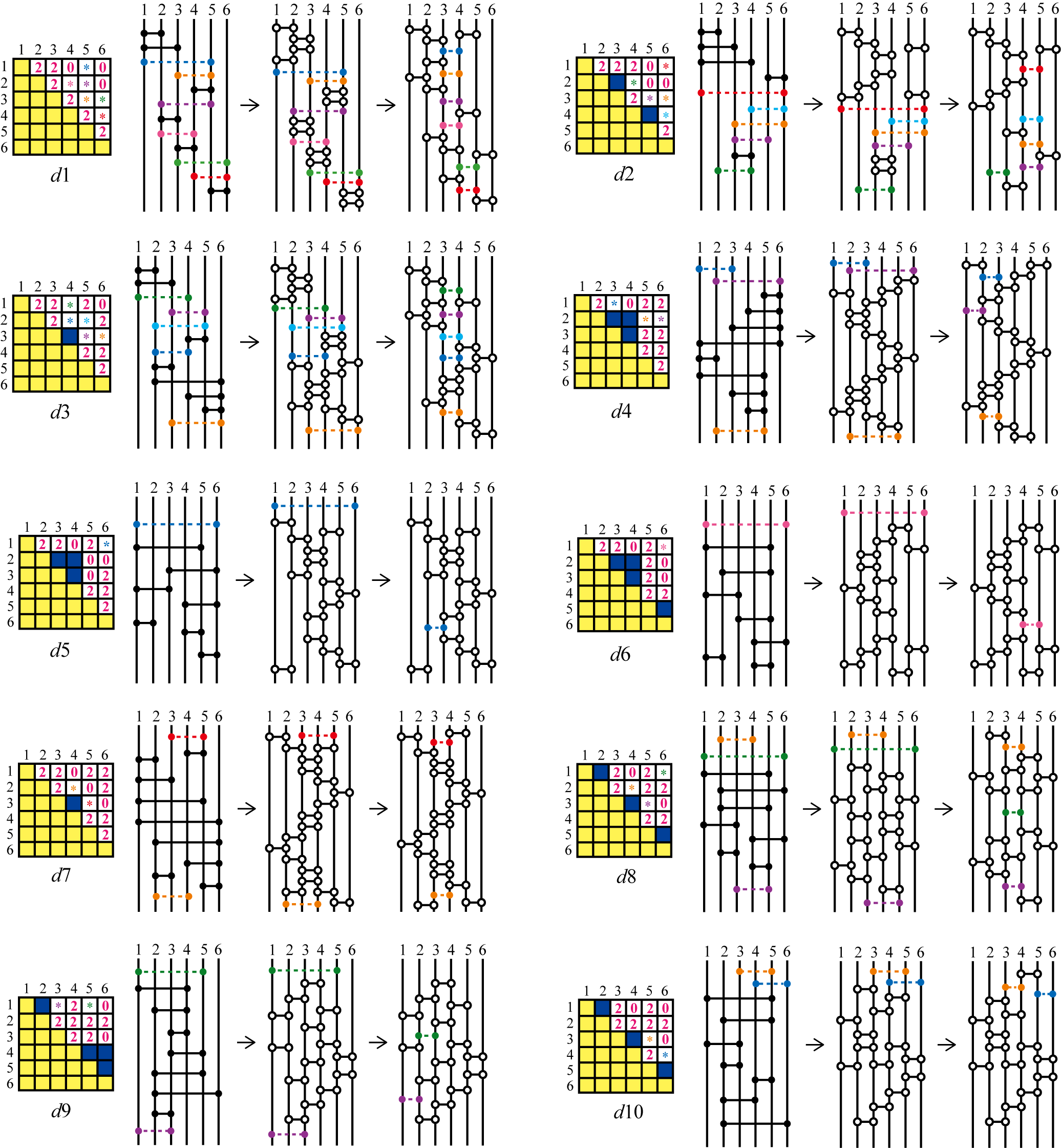}
\caption{The CN-realizable configurations $d1$ to $d10$ that are derived by the ladder diagram.}
\label{fig-d1}
\end{figure}

\begin{figure}[ht]
\centering
\includegraphics[width=13cm]{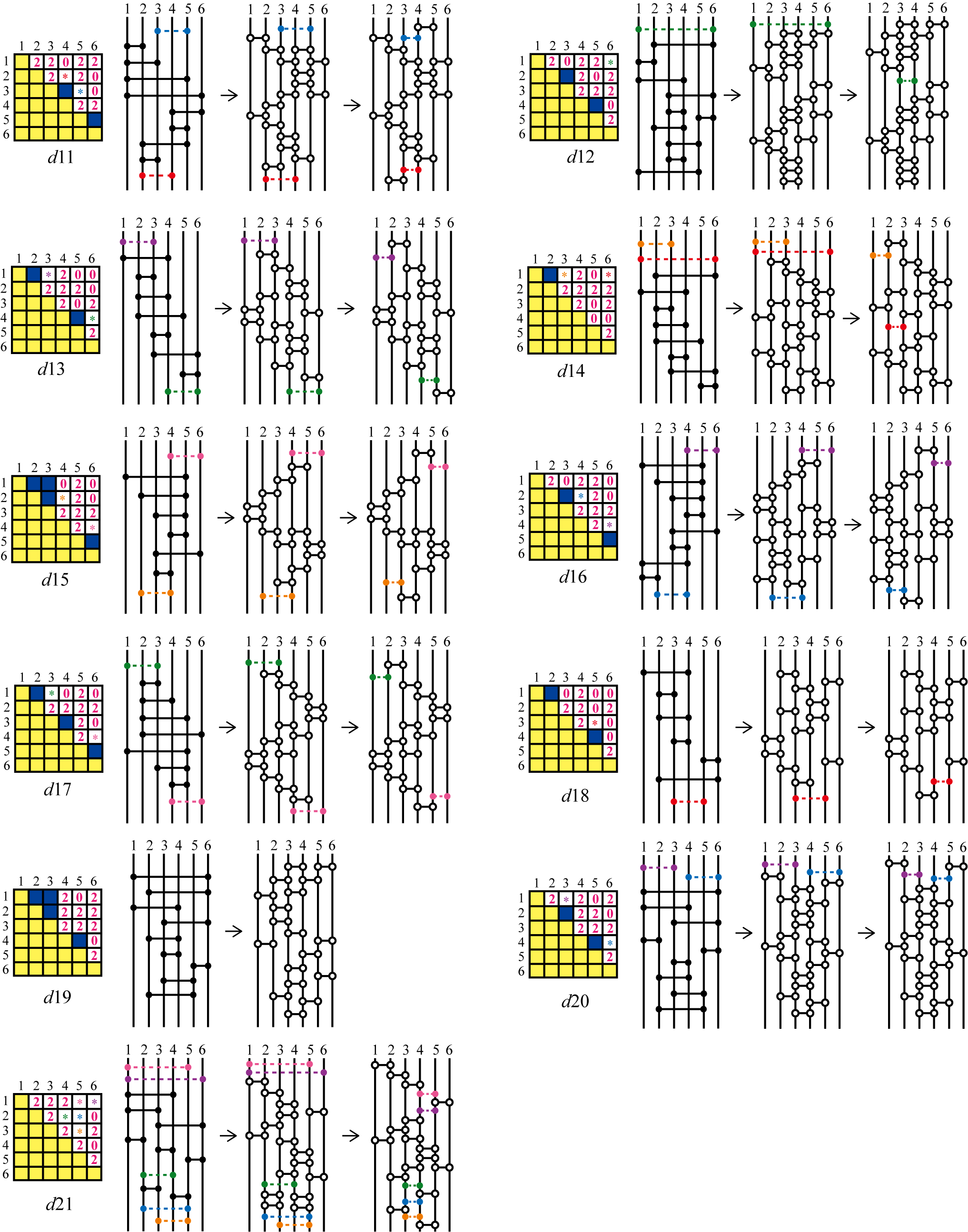}
\caption{The CN-realizable configurations $d11$ to $d21$ that are derived by the ladder diagram.}
\label{fig-d2}
\end{figure}

\end{document}